\documentclass[review,hidelinks,onefignum,onetabnum]{siamart220329_arxiv}

\usepackage{lipsum}
\usepackage{amsfonts}
\usepackage{graphicx}
\usepackage{epstopdf}
\usepackage{algorithmic}
\ifpdf
\DeclareGraphicsExtensions{.eps,.pdf,.png,.jpg}
\else
\DeclareGraphicsExtensions{.eps}
\fi

\usepackage{amsmath}
\usepackage{caption}
\usepackage{subcaption}
\usepackage{bm}
\usepackage{tikz}
\usetikzlibrary{shapes.geometric, arrows}
\usepackage{pgfplots}
\usepackage[version=4]{mhchem}

\usepackage{siunitx}
\usepackage{longtable,tabularx}
\tikzstyle{process} = [rectangle, rounded corners, minimum width=2cm, minimum height=1cm,text centered, draw=black]
\tikzstyle{process_small} = [rectangle, rounded corners, minimum width=0.6cm, minimum height=0.7cm,text centered, draw=black]
\tikzstyle{decision} = [diamond, minimum width=1cm, minimum height=1cm, text centered, draw=black]
\tikzstyle{arrow} = [thick,->,>=stealth]

\newcommand{\norm}[1]{\left\lVert#1\right\rVert}

\newsiamremark{remark}{Remark}
\newsiamremark{hypothesis}{Hypothesis}
\crefname{hypothesis}{Hypothesis}{Hypotheses}
\newsiamthm{claim}{Claim}

\headers{Weighted POD for High-Dimensional Optimization}{S. P. C. van Schie, B. Kramer and J. T. Hwang}

\title{Weighted Proper Orthogonal Decomposition for High-Dimensional Optimization\thanks{{\bfseries Funding:} S. van Schie and J. Hwang were supported by NASA under award no.~80NSSC21M0070. B. Kramer was supported by the Air Force Office of Scientific Research (PM Fahroo) under award no. FA9550-24-1-0237. The authors would like to acknowledge the Multidisciplinary Science and Technology Center of the Aerospace Systems Directorate, Air Force Research Laboratory for funding and supporting this effort.}}

\author{Sebastiaan P. C. van Schie\thanks{Department of Mechanical and Aerospace Engineering, UC San Diego, La Jolla, CA 
	(\email{svanschie@ucsd.edu}, \email{bmkramer@ucsd.edu}, \email{jhwang@ucsd.edu}).}
\and Boris Kramer\footnotemark[2]
\and John T. Hwang\footnotemark[2]}

\usepackage{amsopn}

\ifpdf
\hypersetup{
pdftitle={Weighted Proper Orthogonal Decomposition for High-Dimensional Optimization},
pdfauthor={S. P. C. van Schie and B. Kramer and J. T. Hwang}
}
\fi

\begin{document}

\maketitle

\begin{abstract}
	While proper orthogonal decomposition (POD) is widely used for model reduction, its standard form does not take into account any parametric model structure. Extensions to POD have been proposed to address this, but these either require large amounts of solution data, lack online adaptivity, or have limited approximation accuracy. We circumvent these limitations by instead assigning weights to the snapshot matrix columns, and updating these whenever the model is evaluated at a new point in the parameter space. We derive an \textit{a posteriori} error bound that depends on these snapshot weights, show how these weights can be chosen to tighten the error bound, and present an algorithm to compute the corresponding reduced basis efficiently. We show how this weighted POD approach can be used to naturally generalize the calculation of reduced basis derivatives to situations with multidimensional parameter spaces and snapshots at multiple locations in the parameter space. Lastly, we cover how these approaches can be implemented within an optimization algorithm, without the need for an offline training phase. The proposed weighted POD methods with and without reduced basis derivatives are applied to a gradient-based shell thickness optimization problem with 105 design parameters and a time-dependent partial differential equation. The numerical solutions obtained for this problem attain errors that are several orders of magnitude smaller when using weighted POD than those computed with regular POD and Grassmann manifold interpolation, while having comparable wall times per query and requiring fewer high-dimensional model snapshots to reach an optimal solution.	
\end{abstract}




\section{Introduction}
Proper orthogonal decomposition (POD) is widely used as a model reduction technique for a variety of applications, such as optimal control \cite{bib:Ravindran2000}, discovery and analysis of models for complex physical systems \cite{bib:Berkooz1993,bib:Ly2001}, design optimization \cite{bib:LeGresley2000}, uncertainty quantification \cite{bib:Heinkenschloss2020,bib:Chen2017}, inverse problems \cite{bib:Bialecki2003}, and parameter estimation \cite{bib:Schmidt2013}, to name a few. In part, this popularity is due to its optimal approximation property and ease of computation: POD leverages the singular value decomposition (SVD), which gives reduced-rank approximations to a given data matrix that are optimal in the matrix $2$-norm and Frobenius norm \cite[Theorem 3.6]{bib:Antoulas2005}.

Standard POD does not take into account any parametric model structure. Given a data matrix that contains solution snapshots at multiple points in the parameter space and possibly corresponding to multiple time steps, POD finds the best approximation in the $L_2$ sense to \textit{all} snapshots. This results in a single global reduced basis whose predictive performance is highly dependent on how correlated the solution snapshots are. However, in optimization we only require accurate local reduced bases, rather than the best global reduced basis. Moreover, gathering enough solution snapshots to obtain an accurate global reduced basis for the entire parameter space is prohibitively expensive. \\

Many extensions to improve the accuracy of POD for parametric problems have been proposed. These can be classified as global, local or adaptive POD approaches \cite{bib:Lu2019, bib:Peng2016}. Global methods attempt to construct a single reduced basis that works for the entire parameter space. Such bases can be constructed in various ways. Examples include concatenating local reduced bases corresponding to multiple points in the parameter space, and using adaptive sampling to gather more snapshots in regions of interest of the parameter space \cite{bib:Benner2015}. Local POD approaches divide the parameter, physical, or time domain into multiple subdomains and compute a separate reduced basis for each subdomain. Clustering approaches are a prime example of this \cite{bib:Amsallem2012, bib:Peng2016, bib:Sahyoun2013}: Solution data is divided into clusters, and a POD basis is constructed for each cluster. Each basis is considered valid in its respective cluster. Adaptive POD approaches construct a new reduced basis for each point that is queried in the parameter space. Among these are interpolation approaches, which compute an adaptive POD basis by combining information on the spaces spanned by multiple local bases. One interpolation approach, Grassmann manifold interpolation \cite{bib:Amsallem2008, bib:Goutaudier2023}, formulates an interpolation problem in the tangent space at a reference point on a Grassmann manifold. The resulting POD basis for a new point in the parameter space retains its orthonormality, which is not the case when interpolating reduced bases directly in the parameter space.

Snapshot weighting has been used to construct local or adaptive POD bases, in various ways. First, local POD bases have been constructed on subdomains with a clustering-type approach, with snapshot weights based on the proximity of snapshots to the centroid of each subdomain \cite{bib:Peng2016}. Second, goal-oriented POD bases have been constructed for uncertainty quantification, with snapshots weighted according to a quadrature rule based on the distribution of the parameter inputs \cite{bib:Venturi2019}. Third, weights have been used to deal with snapshots that are scattered irregularly throughout the parameter space \cite{bib:Bistrian2016}. Fourth, when including state derivatives in the snapshot matrix, snapshot weights have been chosen to simultaneously mimic Taylor polynomials and to give less importance to more distant snapshots and derivatives \cite{bib:Carlberg2008,bib:Carlberg2011, bib:Schmidt2013}. And lastly, cubic spline interpolation has been used for inverse design problems to directly predict local steady-state solutions at unseen points in the parameter space \cite{bib:BuiThanh2003}. 

However, there are several aspects to these snapshot weighting approaches that limit their usefulness for optimization problems. They either rely on the acquisition of a large number of solution snapshots in an offline data acquisition phase, do not address how the reduced basis can be updated efficiently online when the snapshot weights are updated, or only allow for the interpolation of bases of the same dimension, thereby limiting its flexibility. Whereas these limitations are not detrimental in certain applications, we aim to use POD for optimization with hundreds of design parameters. Optimization algorithms do not in general possess a natural offline-online decomposition. Moreover, covering high-dimensional parameter spaces with solution snapshots is neither computationally feasible (curse of dimensionality) nor desirable, since we only require snapshots in those regions of the parameter space that are explored by the optimization algorithm \textemdash which are not known \textit{a priori}. As such, it is desirable for any suitable reduced model to possess efficient online adaptivity.

We propose a weighted POD approach, wherein weights are assigned to the columns of the snapshot matrix to assign greater importance to data close to the current design point when computing the SVD. The novelty of this work is threefold:
\begin{enumerate}
	\item We derive an \textit{a posteriori} error bound for weighted POD for parametric models, and show how this error bound can be tightened by assigning new snapshot weights for each query point in the parameter space. We assign the weights based on the proximity of each snapshot to the aforementioned query point. 
	\item We present an algorithm with which the SVD of the weighted snapshot matrix can be computed efficiently online.
	\item We generalize the algorithm of Hay et al. \cite{bib:Hay2009} for using snapshot derivatives to compute the derivatives of the reduced basis vectors, to situations with multidimensional parameter spaces and snapshots at multiple points in the parameter space, in an efficient and natural way, by using snapshot weighting.
\end{enumerate}

We start in \cref{sec:background} with a brief overview of high-dimensional models (HDMs) and the application of POD, followed by an SVD updating algorithm for streaming data. We cover the novel contributions in \cref{sec:adapting_POD}: First the weighted POD approach, then its efficient computational implementation, and lastly the reduced basis vector derivative calculation approach. \Cref{sec:optimization_with_POD} presents how we integrate POD into gradient-based optimization algorithms. In \cref{sec:numerical_experiments}, these novel POD extensions are applied to a design optimization problem constrained by a time-dependent partial differential equation (PDE) with $105$ design parameters.

\section{Background}
\label{sec:background}
Let $\Omega \subset \mathbb{R}^d$ be a bounded spatial domain with $d \in \mathbb{N}$, $I = [0, T]$ a time interval with $0 < T \in \mathbb{R}$ and $\mathbb{P} \subset \mathbb{R}^{n_p}$ a compact parameter space. We start in \cref{subsec:highdim_models} by introducing the HDMs and functions that we attempt to approximate. \Cref{subsec:POD} illustrates the use of POD on HDMs, and reviews the POD basis matrix construction. Lastly, \cref{subsec:incremental_SVD} covers the incremental algorithm to efficiently update the SVD of the snapshot matrix whenever new snapshots are appended to it. Collectively, these steps define the POD approach with a global reduced basis, as well as the foundational blocks for the novel methods we propose in \cref{sec:adapting_POD}.

\subsection{High-dimensional models}
\label{subsec:highdim_models}
Let $\mathbb{X}$ be a real Hilbert space on $\Omega$. We consider the space of measurable time-dependent functions taking values in $\mathbb{X}$,
\begin{equation}
	\mathbb{W}_q^{m}(I, \mathbb{X}):=\\
	\Bigg\{y: I \rightarrow \mathbb{X}\ :
	\left(\sum_{\beta \leq m} \ \int_0^T \norm{D_t^\beta y(t)}_\mathbb{X}^q \ {\rm d} t\right)^{1/q} < \infty  \Bigg\},
\end{equation}
whose time derivatives $D_t^\beta y(t) := \frac{\partial^\beta y}{\partial t^\beta}(t)$ up to order $m$ have $\mathbb{X}$-norms that are $q$-integrable on $I$. We use $\mathbb{X} = \mathbb{W}_p^h(\Omega)$ in the rest of this work. In order to generalize this to parametric models, we consider the space of measurable time- and parameter-dependent functions taking values in $\mathbb{X}$, 
\begin{multline}
	\mathbb{W}_q^{m,l}(\mathbb{P}, I, \mathbb{X}):=\\
	\Bigg\{y: \mathbb{P} \times I \rightarrow \mathbb{X}\ :
	\left(\sum_{|\alpha| \leq l} \sum_{\beta \leq m} \ \int_{\mathbb{P}} \int_0^T \norm{D_{\bm{\mu}}^\alpha D_t^\beta y(\bm{\mu}, t)}_\mathbb{X}^q \ {\rm d}t\ {\rm d}\bm{\mu}\right)^{1/q} < \infty  \Bigg\},
\end{multline}
where $\bm{\mu} \in \mathbb{P}$ and $\alpha := (\alpha_1, \ldots, \alpha_{n_p})$ is an $n_p$-tuple of non-negative integers such that $|\alpha| = \sum_{i=1}^{n_p} \alpha_i$ and
\begin{equation}
	D_{\bm{\mu}}^\alpha y := \left[\frac{\partial}{\partial \bm{\mu}_1}\right]^{\alpha_1}\cdots \left[\frac{\partial}{\partial \bm{\mu}_{n_p}}\right]^{\alpha_{n_p}} y.
\end{equation}
It follows that for any $y \in \mathbb{W}_q^{m,l}(\mathbb{P}, I, \mathbb{X})$, the restriction $y(\bm{\mu}, \cdot ) \in \mathbb{W}_q^{m}(I, \mathbb{X})$ for $\bm{\mu} \in \mathbb{P}$.

In this work, we focus on parameter-dependent dynamical systems of the form
\begin{equation}
	\label{eq:dynamicalsystem}
	\frac{\partial y}{\partial t}(\bm{\mu}, t) = f\left(y(\bm{\mu},t), \bm{\mu}, t \right), \quad y(\bm{\mu}, 0) \in \mathbb{X}.
\end{equation}
Here, $y \in \mathbb{W}_q^{m,l}(\mathbb{P}, I, \mathbb{X})$ and $f(y, \cdot, \cdot) \in \mathbb{W}_q^{m-1,l}(\mathbb{P}, I, \mathbb{X})$. Note that we have omitted the dependence on spatial coordinate $x \in \Omega$ to condense the notation. In practice, $y$ is approximated with a numerical solution to a finite-dimensional approximation of \eqref{eq:dynamicalsystem} at a finite number of time steps. We refer to this numerical solution as the HDM solution. Let $k$ denote the time step index, with $\Delta t > 0$ the constant time step size. The numerical solution at time step $k$ is denoted with
\begin{equation}
	\label{eq:seriesexpansion}
	y^k(\bm{\mu}, x) := \sum_{i=1}^{n_x} a_i^k(\bm{\mu}) \gamma_i(x),
\end{equation}
where we have made the dependence of $y^k$ and $\gamma_i$ on $x$ explicit. The $n_x$ basis functions $\gamma_i$ each have a compact support in $\Omega$ and are defined by the numerical discretization, while the degrees of freedom $a_i^k$ define the vector $\bm{a}^k$. Using \eqref{eq:seriesexpansion} in \eqref{eq:dynamicalsystem} lets us define a residual vector $\bm{r}^k: \mathbb{R}^{n_x} \times \mathbb{P} \rightarrow \mathbb{R}^{n_x}$ for each time step. These residual vectors describe the difference after discretization between the left and right sides of \eqref{eq:dynamicalsystem}; $\bm{r}^k (\bm{a}^k,\bm{\mu}) = \bm{0}$ corresponds to a valid degree of freedom vector $\bm{a}^k$ for which $y^k$ (approximately) satisfies \eqref{eq:dynamicalsystem}. Newton's method provides an iterative approach for obtaining such stationary points, through solving linear systems of the form
\begin{equation}
	\label{eq:linearsystem}
	\frac{\partial \bm{r}^k}{\partial \bm{a}_l^k}(\bm{a}_l^k, \bm{\mu})\Delta \bm{a}_{l+1}^k = - \bm{r}^k(\bm{a}_l^k, \bm{\mu}),
\end{equation}
where $\Delta \bm{a}_{l+1}^k = \bm{a}_{l+1}^k - \bm{a}_l^k$ is an incremental update to the numerical solution and $\frac{\partial \bm{r}^k}{\partial \bm{a}_l^k}(\bm{a}_l^k, \bm{\mu}) \in \mathbb{R}^{n_x \times n_x}$. As covered in \cite{bib:Trefethen}, the cost of using direct solution methods to solve \eqref{eq:linearsystem} scales with $n_x^3$. Whereas the cost of using iterative solution methods is lower for large systems, obtaining accurate solutions can still incur a significant computational cost. For linear systems, equation \eqref{eq:linearsystem} needs to be solved once per time step, whereas nonlinear systems typically require multiple iterations per time step. In addition, many-query applications such as uncertainty quantification and optimization, require solutions of \eqref{eq:linearsystem} corresponding to many parameter values $\left\{\bm{\mu}_1, \bm{\mu}_2, \ldots \right\} \subset \mathbb{P}$.

\subsection{Proper orthogonal decomposition}
\label{subsec:POD}
Rather than directly solving equation \eqref{eq:linearsystem}, POD combined with Galerkin projection computes an approximate solution to \eqref{eq:linearsystem} in a low-dimensional subspace of $\mathbb{R}^{n_x}$. Suppose that we are given a POD basis matrix $\bm{\Phi} \in \mathbb{R}^{n_x \times n_r}$ with $\bm{\Phi}^T\bm{\Phi} = \bm{I}_{n_r}$. We define the approximate solution $y_r^k (\bm{\mu}, x)$ as the projection of $y^k$ onto the low-dimensional subspace spanned by the columns of $\bm{\Phi}$, such that
\begin{equation}
	\label{eq:pod_variable}
	y_r^k (\bm{\mu}, x) := \sum_{i=1}^{n_x} \sum_{j=1}^{n_r} \underbrace{\left(\bm{\Phi}_{ij} a_{i}^k(\bm{\mu})\right)}_{=a_{r,j}^k(\bm{\mu})} \underbrace{\left(\bm{\Phi}_{ij} \gamma_{i}(x)\right)}_{=\gamma_{r,j}(x)} = \left(\bm{\Phi} \bm{a}_r^k(\bm{\mu})\right)^T\left( \bm{\Phi} \bm{\gamma}_r(x)\right),
\end{equation}
with $\bm{a}_r^k \in \mathbb{R}^{n_r}$ the reduced degrees of freedom and $\gamma_{r,j} : \Omega \rightarrow \mathbb{R}$ the associated reduced basis functions for a fixed reduced dimension $n_r$ and $j = 1, 2, \dots, n_r$. The corresponding linear relation between the full-dimensional and reduced degree of freedom vectors then is
\begin{equation}
	\label{eq:pod_dof_projection}
	\bm{a}^k := \bm{\Phi} \bm{a}_{r}^k.
\end{equation}
Instead of directly solving the HDM \eqref{eq:linearsystem} we insert \eqref{eq:pod_dof_projection} into \eqref{eq:linearsystem}. This results in an over-determined linear system; to produce a linear system with a single, unique solution we left-multiply with a test matrix $\bm{\Psi}^T \in \mathbb{R}^{n_r \times n_x}$. This results in the $n_r \times n_r$ system for the POD coefficients,
\begin{equation}
	\label{eq:pod_equation}
	\bm{\Psi}^T \frac{\partial \bm{r}^k}{\partial \bm{a}_l^k} \bm{\Phi} \Delta \bm{a}_{r,l+1}^k = - \bm{\Psi}^T \bm{r}^k(\bm{\Phi} \bm{a}_{r,l}^k).
\end{equation}
Where, as before, $l$ denotes the current iteration number of Newton's method. POD with Galerkin projection is computationally efficient if $n_r \ll n_x$. Often $n_r$ is on the order of tens or hundreds, and thus direct methods can be used to efficiently solve \eqref{eq:pod_equation}. Letting $\bm{\Psi} \neq \bm{\Phi}$ results in a POD-Petrov-Galerkin approach. In this work, we focus instead on POD-Galerkin approaches, in which $\bm{\Psi} = \bm{\Phi}$. Once the low-dimensional solution $\bm{a}_{r,l+1}^k \in \mathbb{R}^{n_r}$ is known, an approximation $\hat{\bm{a}}_{l+1}^k \in \mathbb{R}^{n_x}$ to the solution $\bm{a}_{l+1}^k$ of \eqref{eq:linearsystem} can be computed with \eqref{eq:pod_dof_projection}. If $\bm{\Phi}$ is sufficiently accurate, $\hat{\bm{a}}_{l+1}^k \approx \bm{a}_{l+1}^k$.

A computational bottleneck is that the matrix $\frac{\partial \bm{r}^k}{\partial \bm{a}_l^k}$ and the right-side vector $\bm{r}^k$ still need to be assembled in order to solve the reduced equation \eqref{eq:pod_equation}. The costs of assembling these terms scale with $n_x$, and can thus be a significant part of the total cost of solving \eqref{eq:linearsystem} when $n_r \ll n_x$. Approximate assembly approaches called hyper-reduction methods can be used to speed up matrix and vector assembly. Applications of such methods within the context of optimization can be found for example in \cite{bib:Amsallem2015,bib:vanSchie2023}.

The POD basis matrix $\bm{\Phi}$ is constructed based on solution data of the HDM \eqref{eq:linearsystem}. Suppose we have available $n_s$ sets of solution snapshots $\bm{A}_j(\bm{\mu}_j) \in \mathbb{R}^{n_x \times n_t}$, such that $\bm{A}_j (\bm{\mu}_j) := \begin{bmatrix} \bm{a}_j^1(\bm{\mu}_j) & \bm{a}_j^2(\bm{\mu}_j) & \cdots & \bm{a}_j^{n_t}(\bm{\mu}_j) \end{bmatrix}$ at various design parameter vectors $\bm{\mu}_j \in \mathbb{P}$ for $j = 1, 2, \ldots, n_s$. We gather these in the snapshot matrix 
\begin{equation}
	\label{eq:snapshotmatrix}
	\bm{A} := \begin{bmatrix} \bm{A}_1(\bm{\mu}_1) & \bm{A}_2(\bm{\mu}_2) & \cdots & \bm{A}_{n_s}(\bm{\mu}_{n_s}) \end{bmatrix} \in \mathbb{R}^{n_x \times n_s n_t}.
\end{equation}
Let $\bm{E} \in \mathbb{R}^{n_x \times n_x}$ denote the symmetric positive-definite inner product matrix corresponding to inner product $\langle \cdot , \cdot \rangle_{\mathbb{X}}$, with $\bm{E} = \bm{E}^{1/2} \bm{E}^{1/2}$. The POD basis matrix $\bm{\Phi}$ can be found by first computing the thin SVD,
\begin{equation}
	\label{eq:standard_svd}
	\bm{E}^{1/2} \bm{A} = \bm{U} \bm{\Sigma} \bm{V}^T,
\end{equation}
where $\bm{U} = \begin{bmatrix} \bm{u}_1 & \cdots & \bm{u}_{n_s n_t} \end{bmatrix} \in \mathbb{R}^{n_x \times n_s n_t}$ consists of the left singular vectors as its columns, $\bm{\Sigma}$ is a matrix with nonzero singular values $\{\sigma_i\}_{i=1}^{n_s n_t}$ on its diagonal in decreasing order and $\bm{V} = \begin{bmatrix} \bm{v}_1 & \cdots & \bm{v}_{n_s n_t} \end{bmatrix} \in \mathbb{R}^{n_s n_t \times n_s n_t}$ contains the right singular vectors as its columns. The matrices $\bm{U}$ and $\bm{V}$ each have orthonormal columns. The POD basis matrix $\bm{\Phi}$ is constructed by taking the first $n_r$ columns of $\bm{U}$ and left-multiplying with $\bm{E}^{-1/2}$, such that
\begin{equation}
	\bm{\Phi} := \begin{bmatrix} \bm{\phi}_1 & \bm{\phi}_2 \cdots & \bm{\phi}_{n_r} \end{bmatrix} = \bm{E}^{-1/2}\begin{bmatrix} \bm{u}_1 & \bm{u}_2 \cdots & \bm{u}_{n_r} \end{bmatrix}.
\end{equation}
These columns of $\bm{U}$ correspond to the largest singular values of $\bm{E}^{1/2}\bm{A}$ and together span the $n_r$-dimensional column space that best approximates $\bm{E}^{1/2}\bm{A}$. Moreover, the POD modes $\left\{\bm{\phi}_i \right\}_{i=1}^{n_r}$ are orthonormal with respect to inner product $\langle \cdot , \cdot \rangle_{\mathbb{X}}$.

\subsection{Incremental singular value decompositions}
\label{subsec:incremental_SVD}
Since we do not know \textit{a priori} which regions of the parameter space are of interest when solving an optimization problem, new solution snapshots are acquired while running the optimization algorithm. Useful model reduction approaches for optimization thus require an efficient way to update the SVD and, hence, the POD basis, whenever new data is appended to the snapshot matrix. Moreover, we want to control the size of the stored snapshot matrix to manage its memory footprint. To simultaneously allow for both of these requirements, we use an incremental SVD approach based on \cite{bib:Brand2003, bib:Iwen2016}. Suppose that the snapshot matrix $\bm{X} \in \mathbb{R}^{n_x \times n_s n_t}$ (with SVD $\bm{U}_{\bm{X}} \bm{\Sigma}_{\bm{X}} \bm{V}_{\bm{X}}^T$) contains data from $n_s$ parametric solutions, and we obtain a new matrix $\bm{Y} \in \mathbb{R}^{n_x \times n_t}$ (with its own SVD $\bm{U}_{\bm{Y}} \bm{\Sigma}_{\bm{Y}} \bm{V}_{\bm{Y}}^T$) of time trajectories for a specific parameter value. \Cref{alg:svd_update} covers the incremental SVD algorithm, which computes the SVD of $\begin{bmatrix} \bm{X} & \bm{Y} \end{bmatrix} \in \mathbb{R}^{n_x \times (n_s + 1) n_t}$ based on the separate SVDs of $\bm{X}$ and $\bm{Y}$. In the rest of this work we truncate all SVD components that correspond to singular values smaller than $\epsilon_{\text{SVD}} = 10^{-8}$. 

\begin{algorithm}
	\caption{Incremental SVD algorithm} 
	\label{alg:svd_update}
	\begin{algorithmic}[1]
		\REQUIRE{ SVD of snapshot matrix $\bm{X} = \bm{U}_{\bm{X}} \bm{\Sigma}_{\bm{X}} \bm{V}_{\bm{X}}^T \in \mathbb{R}^{n_x \times n_s n_t}$, SVD of appended snapshot $\bm{Y} = \bm{U}_{\bm{Y}} \bm{\Sigma}_{\bm{Y}} \bm{V}_{\bm{Y}}^T \in \mathbb{R}^{n_x \times n_t}$, orthonormality tolerance $\epsilon_{\bot} \in \mathbb{R}$ }
		\ENSURE{SVD of extended snapshot matrix $\begin{bmatrix} \bm{X} & \bm{Y} \end{bmatrix} = \bm{U} \bm{\Sigma} \bm{V}^T\in$ $\mathbb{R}^{n_x \times (n_s) n_t}$}
		
		\STATE{Compute orthogonal complement of $\bm{U}_{\bm{Y}}$ w.r.t. $\bm{U}_{\bm{X}}$: $\bm{U}_{\bm{Y}, \bot} := \bm{U}_{\bm{Y}} - \bm{U}_{\bm{X}} \bm{U}_{\bm{X}}^T \bm{U}_{\bm{Y}}$}
		\STATE{Compute QR-decomposition of $\bm{U}_{\bm{Y}, \bot}$: $\bm{U}_{\bm{Y}, \bot} = \bm{Q}\bm{R}$}
		\STATE{Compute $\bm{G} := \begin{bmatrix}
				\bm{\Sigma}_{\bm{X}} & \bm{U}_{\bm{X}}^T \bm{U}_{\bm{Y}} \bm{\Sigma}_{\bm{Y}} \\
				& \bm{R} \bm{\Sigma}_{\bm{Y}}
			\end{bmatrix}$}
		\STATE{Compute SVD: $\bm{G} = \bm{U}_{\bm{G}} \bm{\Sigma}_{\bm{G}} \bm{V}_{\bm{G}}^T$}
		\STATE{Set $\bm{U} := \begin{bmatrix}
				\bm{U}_{\bm{X}} & \bm{Q}
			\end{bmatrix} \bm{U}_{\bm{G}}$, $\bm{\Sigma} := \bm{\Sigma}_{\bm{G}}$, $\bm{V} := \begin{bmatrix}
				\bm{V}_{\bm{X}} & \\
				& \bm{V}_{\bm{Y}}
			\end{bmatrix} \bm{V}_{\bm{G}}$}
		\WHILE{$\max\left(\text{diag}\left(\bm{U}^T \bm{U} \right) \right) > 1 + \epsilon_{\bot} \  \textbf{or} \ \min\left(\text{diag}\left(\bm{U}^T \bm{U} \right) \right) < 1 - \epsilon_{\bot}$}
		\STATE{ Compute QR-decomposition of $\begin{bmatrix}
				\bm{X} & \bm{Y}
			\end{bmatrix} \bm{V}$: $\begin{bmatrix}
				\bm{X} & \bm{Y}
			\end{bmatrix} \bm{V} = \tilde{\bm{Q}}\tilde{\bm{R}}$}
		\STATE{ Compute SVD of $\tilde{\bm{Q}}^T\begin{bmatrix}
				\bm{X} & \bm{Y}
			\end{bmatrix}$: $\tilde{\bm{Q}}^T\begin{bmatrix}
				\bm{X} & \bm{Y}
			\end{bmatrix} = \hat{\bm{U}}\hat{\bm{\Sigma}} \hat{\bm{V}}^T$}
		\STATE{Set $\bm{U} := \tilde{\bm{Q}} \hat{\bm{U}}$, $\bm{\Sigma} := \hat{\bm{\Sigma}}$, $\bm{V} := \hat{\bm{V}}$}
		\ENDWHILE
	\end{algorithmic}
\end{algorithm}

Let $n_{\bm{X}}$ denote the number of retained SVD components of matrix $\bm{X}$, and $n_{\bm{Y}}$ the same for $\bm{Y}$. Whereas the cost of directly computing the SVD of $\begin{bmatrix}
	\bm{X} & \bm{Y} \end{bmatrix}$ is $\mathcal{O}\left(\min(n_x(n_s + 1)^2 n_t^2, n_x^2(n_s + 1) n_t ) \right)$ \cite{bib:Golub}, the dominant operations in \Cref{alg:svd_update} consist of a QR-decomposition of a matrix of size $n_x \times n_{\bm{Y}}$, at a cost of $\mathcal{O}\left(n_x n_{\bm{Y}}^2) \right)$, and an SVD of a matrix of size $(n_{\bm{X}} + n_{\bm{Y}}) \times (n_{\bm{X}} + n_{\bm{Y}})$, at a cost of $\mathcal{O}\left((n_{\bm{X}} + n_{\bm{Y}})^3 \right)$.
These operations make use of the SVD of snapshot matrix $\bm{Y}$, which comes at a cost of $\mathcal{O}\left(n_x n_t^2\right)$ as long as $n_t \leq n_x$.

Repeated application of \Cref{alg:svd_update} can lead to accumulating round-off errors. This in turn results in the loss of orthonormality of the singular vectors. To remedy this, we include an iterative correction step that is based on the description in \cite{bib:Vasudevan2019}. Each iteration of the correction step computes a QR-decomposition and an SVD of a smaller matrix. When the $2$-norm of one of the left-singular vectors deviates from $1$ by more than $\epsilon_{\bot}$, we keep iteratively applying the aforementioned correction step until the norm deviation is smaller than $\epsilon_{\bot}$. Each iteration of the correction step requires a QR-decomposition of a matrix of size $n_x \times (n_{\bm{X}} + n_{\bm{Y}})$, at a cost of $\mathcal{O}\left(n_x (n_{\bm{X}} + n_{\bm{Y}})^2 \right)$ provided that $n_{\bm{X}} + n_{\bm{Y}} \leq n_x$, and an SVD of a matrix of size $(n_{\bm{X}} + n_{\bm{Y}}) \times (n_{\bm{X}} + n_{\bm{Y}})$, at a cost of
$\mathcal{O}\left((n_{\bm{X}} + n_{\bm{Y}})^3 \right)$. In practice, we find that using $\epsilon_{\bot} = 10^{-2}$ gives good results, and requires only one or two correction iterations whenever the norms of the reduced basis vectors deviate from one by more than $\epsilon_{\bot}$.

\section{Adapting parametric POD with snapshot weighting and reduced basis derivatives}
\label{sec:adapting_POD}
While POD is a widely-used model reduction approach, it does not take the parametric structure of the HDM into account. This can lead to poor prediction accuracy unless a large number of basis vectors are used. We introduce two ways in which POD can be made more robust and accurate with respect to parameter changes. The first of these is covered in \cref{subsec:weighted_POD}: We show how the POD approximation error at query point $\hat{\bm{\mu}} \in \mathbb{P}$ can be bounded from above for a set of weighted snapshots, and how this novel upper bound can be minimized by assigning the snapshot weights inversely proportional to the distance between each snapshot and $\hat{\bm{\mu}}$. \Cref{subsec:snapshot_weighting} covers an efficient algorithm to compute the POD basis $\bm{\Phi}$ from the weighted snapshot matrix. The second approach is covered in \cref{subsec:state_sensitivities} and allows us to leverage the derivatives of (state) snapshots with respect to the design parameters to compute the corresponding derivatives of the reduced basis. These can make the reduced basis more robust with respect to parameter changes.

\subsection{Weighted POD}
\label{subsec:weighted_POD}
We define the POD approximation error at a new query point $\hat{\bm{\mu}} \in \mathbb{P}$, and prove an \textit{a posteriori} bound for this error. From this bound follows the approach with which we assign weights to the snapshots contained in $\bm{A}$ defined in \eqref{eq:snapshotmatrix}. Let $\left\{y^k(\hat{\bm{\mu}})\right\}_{k=1}^{n_t}$ denote the HDM solution $y(\hat{\bm{\mu}}, t)$ at each of the time steps, and let $\left\{y_r^k(\hat{\bm{\mu}})\right\}_{k=1}^{n_t}$ denote the corresponding approximate solution $y_r(\hat{\bm{\mu}}, t)$ that is obtained through the use of POD, where we have again omitted dependencies on the spatial coordinate $x \in \Omega$. We express the POD approximation error in the squared $\mathbb{W}_2^1(I,\mathbb{X})$-norm and approximate it in the $\mathbb{X}$-norm by using the trapezoid rule, as in \cite{bib:Kunisch2003}. This leads to the weighted sum
\begin{equation}
	\label{eq:time_integral_norm}
	\norm{y(\hat{\bm{\mu}}, t) - y_r(\hat{\bm{\mu}}, t)}_{\mathbb{W}_2^1(I,\mathbb{X})}^2 \approx \sum_{k=1}^{n_t} \omega^k \norm{y^k(\hat{\bm{\mu}}) - y_r^k(\hat{\bm{\mu}})}_{\mathbb{X}}^2,
\end{equation}
where $\left\{\omega_k \right\}_{k=1}^{n_t}$ are time integration weights. We recall that $\mathbb{X} = \mathbb{W}_p^h(\Omega)$. We derive an \textit{a posteriori} error bound for \eqref{eq:time_integral_norm}, where we consider the possibility of assigning different time integration weights $\left\{\omega_k \right\}_{k=1}^{n_t}$ and snapshot weights $\left\{\omega_j \right\}_{j=1}^{n_s}$ when computing the POD basis matrix $\bm{\Phi}$.

\begin{remark}
	If (approximate) time derivative data is included in the snapshot matrix, $\mathbb{W}_2^m(I,\mathbb{X})$-norms with $m > 1$ can be used instead. We do not consider such situations here; interested readers are referred to \cite{bib:Iliescu2014, bib:Koc2021}.   
\end{remark}

\begin{theorem}[weighted POD error bound]\label{th:errorbound}
	Let $\left\{\left\{y^k(\bm{\mu}_j) \right\}_{j=1}^{n_s} \right\}_{k=1}^{n_t}$ denote a set of given solution snapshots to equation \eqref{eq:linearsystem}, and let $\left\{\omega_j \right\}_{j=1}^{n_s}$ be a set of snapshot weights that form a partition of unity. Suppose that the POD basis $\left\{\bm{\phi}_i \right\}_{i=1}^{n_r} = \left\{\bm{E}^{-1/2}\overline{\bm{\phi}}_i \right\}_{i=1}^{n_r}$ corresponds to the largest $n_r$ singular values $\left\{\sigma_i \right\}_{i=1}^{n_r}$ and associated left-singular vectors $\left\{\overline{\bm{\phi}}_i \right\}_{i=1}^{n_r}$ of $\bm{X}_{\hat{\bm{\mu}}} = \bm{X}\bm{W}_{\hat{\bm{\mu}}}^{1/2} = \bm{E}^{1/2}\bm{A}\bm{W}_k^{1/2}\bm{W}_{\hat{\bm{\mu}}}^{1/2}$, where $\bm{E}$ is the discretized $\mathbb{X}$-inner product matrix, $\bm{A}$ the snapshot matrix defined in \eqref{eq:snapshotmatrix}, and $\bm{W}_k$ and $\bm{W}_{\hat{\bm{\mu}}}$ are matrices that contain the weights $\omega^k$ and $\omega_j^2$ on their diagonals respectively. Then the time-weighted POD approximation error for a new query point $\hat{\bm{\mu}} \notin \left\{\bm{\mu}_j \right\}_{j=1}^{n_s}$ expressed in the squared $\mathbb{X}$-norm is bounded as 
	\begin{multline}
		\label{eq:PODerrorbound}
		\sum_{k=1}^{n_t} \omega^k \norm{y^k(\hat{\bm{\mu}}) - y_r^k(\hat{\bm{\mu}})}_{\mathbb{X}}^2 \leq 3 n_s \sum_{i= n_r + 1}^{n_s n_t} \sigma_i^2\\
		+ 3 n_s n_p \sum_{k=1}^{n_t} \omega^k \sum_{j=1}^{n_s}  \omega_j^2 \left[\norm{\hat{\bm{\mu}} - \bm{\mu}_j + \rho_j^{1,k} \bm{1}}_{\bm{C}_j^{k,1}}^2 + \norm{\hat{\bm{\mu}} - \bm{\mu}_j + \rho_j^{3,k} \bm{1}}_{\bm{C}_j^{k,3}}^2\right],
	\end{multline}
	where $\left\{\left\{ \rho_j^{1,k}\right\}_{k=1}^{n_t}\right\}_{j=1}^{n_s}$, $\left\{\left\{ \rho_j^{3,k}\right\}_{k=1}^{n_t}\right\}_{j=1}^{n_s}$ are the sets of (positive) radii of open balls centered around $\bm{\mu}_j$ for each index $j$, and $\bm{1} \in \mathbb{R}^{n_p}$ is a vector whose entries are all equal to $1$.
\end{theorem}
\begin{corollary}
	\label{cor:globalPODbound}
	Setting each snapshot weight $\omega_j$ equal to $1/n_s$ results in an error bound for the global POD basis,
	\begin{multline}
		\label{eq:globalPODbound}
		\sum_{k=1}^{n_t} \omega^k \norm{y^k(\hat{\bm{\mu}}) - y_r^k(\hat{\bm{\mu}})}_{\mathbb{X}}^2 \leq \frac{3}{n_s} \sum_{i= n_r + 1}^{n_s n_t} \sigma_i^2 \\
		+ \frac{3 n_p}{n_s} \sum_{k=1}^{n_t} \omega^k \sum_{j=1}^{n_s}  \left[\norm{\hat{\bm{\mu}} - \bm{\mu}_j + \rho_j^{1,k} \bm{1}}_{\bm{C}_j^{k,1}}^2 + \norm{\hat{\bm{\mu}} - \bm{\mu}_j + \rho_j^{3,k} \bm{1}}_{\bm{C}_j^{k,3}}^2\right],
	\end{multline}
	where $\left\{\sigma_i \right\}_{i=1}^{n_s n_r}$ are the singular values of $\frac{1}{n_s}\bm{X}$.
\end{corollary}
\begin{proof}[Proof of Corollary \ref{cor:globalPODbound}]
	In the case $\omega_j = 1/ n_s$, the factor $n_s$ on the right side of \eqref{eq:PODerrorbound} is canceled by the factors $1/n_s^2$, which directly gives \eqref{eq:globalPODbound}.
\end{proof}
We introduce some relevant concepts that are used in the proof of Theorem~\ref{th:errorbound}. These are taken from \cite[Ch.~4]{bib:Brenner2008} and are stated here for convenience and completeness.
\begin{definition}[averaged Taylor polynomial]
	Suppose $y \in \mathbb{W}_2^{1,1}(I,\mathbb{P},\mathbb{X})$, such that $y^k(\bm{\mu})$ is its time-discrete realization for time step $k$ at $\bm{\mu} \in \mathbb{P}$ with $\mathbb{P}$ convex. Let $T_{\overline{\bm{\mu}}}^g( y^k, \bm{\mu})$ denote the Taylor polynomial of $y^k$ centered around $\overline{\bm{\mu}} \in \mathbb{P}$ using derivatives up to (but not including) order $g$, evaluated in $\bm{\mu}$. Let $B(\bm{\mu}_0, \rho) = \left\{\bm{\mu} \in \mathbb{P} \ : \ \norm{\bm{\mu} - \bm{\mu}_0}_2 < \rho \right\}$ be an open ball of radius $\rho > 0$ centered around $\bm{\mu}_0 \in \mathbb{P}$, and let $\eta(\bm{\mu}) \in C_0^\infty (\mathbb{P})$ be a function such that $\text{supp}(\eta) = \overline{B}$ and $\int_{\mathbb{P}} \eta(\bm{\mu})\ \rm{d}\bm{\mu} = 1$. Then the Taylor polynomial of $y^k$ of order $g$ averaged over $B$ is
	\begin{equation}
		Q^g (y^k, \bm{\mu}) := \int_B \eta (\overline{\bm{\mu}}) T_{\overline{\bm{\mu}}}^g( y^k, \bm{\mu})\ {\rm d}\overline{\bm{\mu}}.
	\end{equation}
	The pointwise approximation error of the Taylor polynomial is defined as
	\begin{equation}
		R^{g,k}(\bm{\mu}) := y^k(\bm{\mu}) - Q^g (y^k, \bm{\mu}),
	\end{equation}
	\end{definition}
Since $y \in \mathbb{W}_2^{1,1}(I,\mathbb{P},\mathbb{X})$, it has weak derivatives with respect to coordinates of $\mathbb{P}$ almost everywhere. Thus $Q^g (y^k, \bm{\mu})$ exists for all $\bm{\mu} \in \mathbb{P}$.
\begin{proposition}[approximation error of averaged Taylor polynomials]
	Suppose $\mathbb{P}$ is convex, $y \in \mathbb{W}_2^{1,1}(I,\mathbb{P},\mathbb{X})$, and let $y^k$ be one of its time-discrete realizations. Let $Q^g (y^k, \bm{\mu})$ denote the averaged Taylor polynomial of $y^k$ evaluated in $\bm{\mu} \in \mathbb{P}$. The pointwise approximation error of the Taylor polynomial, $R^{g,k}(\bm{\mu})$, can be written as
	\begin{equation}
		\begin{aligned}
			R^{g,k}(\bm{\mu}) = & \int_B \eta (\overline{\bm{\mu}}) \left[ \sum_{|\alpha| = g} (\bm{\mu} - \overline{\bm{\mu}})^\alpha \int_0^1 \frac{z^{g-1}}{\alpha !} D^\alpha y^k\left(\bm{\mu} + z (\overline{\bm{\mu}} - \bm{\mu})\right)\ {\rm d}z \right]\ {\rm d}\overline{\bm{\mu}}.
		\end{aligned}
	\end{equation}
	For $g = 1$ we use the more compact notation
	\begin{equation}
		\begin{aligned}
			R^{1,k}(\bm{\mu})  = & \int_B \eta (\overline{\bm{\mu}}) \left[ \sum_{\alpha = 1}^{n_p} (\mu_\alpha - \overline{\mu}_\alpha) \int_{\bm{\mu}}^{\overline{\bm{\mu}}} D^\alpha y^k\left(\bm{z} )\right)\ {\rm d}\bm{z} \right]\ {\rm d}\overline{\bm{\mu}},
		\end{aligned}
	\end{equation}
	with $\bm{z} = \bm{\mu} + z (\overline{\bm{\mu}} - \bm{\mu})$, such that $z \in [0, 1]$ and $\bm{z} \in \mathbb{P}$.
\end{proposition}
The proposition above shows that the approximation error $R^{g,k}(\bm{\mu})$ is equal to the integral of the leading truncated Taylor polynomial terms between $\bm{\mu}$ and each $\overline{\bm{\mu}} \in B \subseteq \mathbb{P}$, weighted according to cut-off function $\eta$.
\begin{remark}
	The radii $\left\{\left\{ \rho_j^{1,k}\right\}_{k=1}^{n_t}\right\}_{j=1}^{n_s}$, $\left\{\left\{ \rho_j^{3,k}\right\}_{k=1}^{n_t}\right\}_{j=1}^{n_s}$ in \Cref{th:errorbound} can be taken as the limit to zero for functions that are continuous and differentiable almost everywhere. Doing this recovers the non-averaged Taylor polynomial-based error bound.
\end{remark}
Having defined the approximation error, we can prove Theorem~\ref{th:errorbound}.
\begin{proof}[Proof of Theorem \ref{th:errorbound}]
	Application of the triangle inequality and the Cauchy-Schwarz inequality results in
	\begin{equation}
		\label{eq:triangle_inequality}
		\begin{aligned}            
			\norm{y^k(\hat{\bm{\mu}}) - y_r^k(\hat{\bm{\mu}})}_{\mathbb{X}}^2 & \leq 3\underbrace{\norm{y^k(\hat{\bm{\mu}}) - \sum_{j=1}^{n_s} \omega_j y^k(\bm{\mu}_j)}_{\mathbb{X}}^2}_{\text{Term 1}}\\
			& + 3\underbrace{\norm{\sum_{j=1}^{n_s} \omega_j y^k(\bm{\mu}_j) - \sum_{j=1}^{n_s} \sum_{i=1}^{n_r} \langle \omega_j y^k(\bm{\mu}_j), \phi_i \rangle_{\mathbb{X}} \phi_i}_{\mathbb{X}}^2}_{\text{Term 2}}\\
			& + 3\underbrace{ \norm{\sum_{j=1}^{n_s} \sum_{i=1}^{n_r} \langle \omega_j y^k(\bm{\mu}_j), \phi_i \rangle_{\mathbb{X}} \phi_i - y_r^k(\hat{\bm{\mu}})}_{\mathbb{X}}^2}_{\text{Term 3}}.
		\end{aligned}
	\end{equation}
	Term 1 consists of the difference between $y^k(\hat{\bm{\mu}})$ and linear combinations (with weights $\left\{\omega_j \right\}_{j=1}^{n_s}$) of the solution snapshots of the same time step at $\left\{\bm{\mu}_j \right\}_{j=1}^{n_s}$. Term 2 consists of the weighted POD projection of the solution snapshots, whereas term 3 describes the difference between the POD-projected snapshots and the POD solution at $\hat{\bm{\mu}}$.  
	
	We start by bounding term 1 from above. We note that each snapshot is a zeroth-order Taylor polynomial for $y^k(\hat{\bm{\mu}})$. Hence the difference $y^k(\hat{\bm{\mu}}) - y^k(\bm{\mu}_j)$ is equal to 
	\begin{equation}
		R_j^{1,k}(\hat{\bm{\mu}}) = \int_{B_j^k} \eta (\overline{\bm{\mu}}) \left[ \sum_{\alpha = 1}^{n_p} (\hat{\mu}_{\alpha} - \overline{\mu}_{\alpha}) \int_{\hat{\bm{\mu}}}^{\overline{\bm{\mu}}} \frac{\partial}{\partial \overline{\mu}_{\alpha}} y^k\left(\bm{z}\right)\ {\rm d}\bm{z} \right]\ {\rm d}\overline{\bm{\mu}}.
	\end{equation}
	We pick $B_j^k = B(\bm{\mu}_j, \rho_j^k)$. Using the fact that the snapshot weights $\left\{\omega_j \right\}_{j=1}^{n_s}$ form a partition of unity then allows us to conclude that
	\begin{equation}
		\begin{aligned}
			y^k(\hat{\bm{\mu}}) - \sum_{j=1}^{n_s} \omega_j y^k(\bm{\mu}_j) & = \sum_{j=1}^{n_s} \omega_j \left(y^k(\hat{\bm{\mu}}) - y^k(\bm{\mu}_j) \right) = \sum_{j=1}^{n_s} \omega_j R_j^{1,k}(\hat{\bm{\mu}}).
		\end{aligned}
	\end{equation}
	Then, by the Cauchy-Schwarz inequality
	\begin{equation}
		\norm{y^k(\hat{\bm{\mu}}) - \sum_{j=1}^{n_s} \omega_j y^k(\bm{\mu}_j)}_{\mathbb{X}}^2 = \norm{\sum_{j=1}^{n_s} \omega_j R_j^{1,k}(\hat{\bm{\mu}})}_{\mathbb{X}}^2 \leq n_s \sum_{j=1}^{n_s} \omega_j^2 \norm{R_j^{1,k}}_{\mathbb{X}}^2.
	\end{equation}
	We note that $\hat{\mu}_{\alpha} - \overline{\mu}_{\alpha} \leq \hat{\mu}_{\alpha} - \mu_{j,\alpha} + \rho_j^k$ for all components $\alpha$, since $B_j^k$ is centered around $\bm{\mu}_j$. It follows that
	\begin{equation}
		\label{eq:term1_bound}
		\begin{aligned}
			\norm{R_j^{1,k}}_{\mathbb{X}}^2 \leq & \norm{\sum_{\alpha = 1}^{n_p} \left(\hat{\mu}_{\alpha} - \mu_{j,\alpha} + \rho_j^k \right) \int_{B_j^k} \eta (\overline{\bm{\mu}}) \left[ \int_{\hat{\bm{\mu}}}^{\overline{\bm{\mu}}} \frac{\partial}{\partial \overline{\mu}_\alpha} y^k\left(\bm{z}\right)\ {\rm d}\bm{z} \right]\ {\rm d}\overline{\bm{\mu}}}_{\mathbb{X}}^2\\
			\leq & n_p \sum_{\alpha = 1}^{n_p} \left(\hat{\mu}_{\alpha} - \mu_{j,\alpha} + \rho_j^k \right)^2 \norm{\int_{B_j^k} \eta (\overline{\bm{\mu}}) \left[ \int_{\hat{\bm{\mu}}}^{\overline{\bm{\mu}}} \frac{\partial}{\partial \overline{\mu}_\alpha} y^k\left(\bm{z}\right)\ {\rm d}\bm{z} \right]\ {\rm d}\overline{\bm{\mu}}}_{\mathbb{X}}^2 \\
			\leq & n_p \sum_{\alpha = 1}^{n_p} \left(\hat{\mu}_{\alpha} - \mu_{j,\alpha} + \rho_j^k \right)^2 c_{j,\alpha}^{k,1}.
		\end{aligned}
	\end{equation}   
	The third step in \eqref{eq:term1_bound} follows from the fact that $y^k$ and its derivatives are integrable on $\mathbb{P}$, allowing us to bound the $\mathbb{X}$-norm from above with the constants $c_{j,\alpha}^k > 0$. The result can be written as a weighted inner product
	\begin{equation}
		\begin{aligned}
			\norm{R_j^{1,k}}_{\mathbb{X}}^2 \leq & n_p \left(\hat{\bm{\mu}} - \bm{\mu}_j + \rho_j^k \bm{1} \right)^T\underbrace{\begin{bmatrix}
					c_{j,1}^{k,1} & & \\
					& c_{j,2}^{k,1} & \\
					& & \ddots
			\end{bmatrix}}_{=\bm{C}_j^{k,1}} \left(\hat{\bm{\mu}} - \bm{\mu}_j + \rho_j^k \bm{1} \right) \\
			\leq & n_p \norm{\hat{\bm{\mu}} - \bm{\mu}_j + \rho_j^k \bm{1}}_{\bm{C}_j^{k,1}}^2,
		\end{aligned}
	\end{equation}   
	where $\norm{\cdot }_{\bm{C}_j^{k,1}}$ is used to denote a norm weighted according to $\bm{C}_j^{k,1}$. We conclude from the preceding that term 1 is bounded from above as
	\begin{equation}
		\sum_{k=1}^{n_t} \omega^k \norm{y^k(\hat{\bm{\mu}}) - \sum_{j=1}^{n_s} \omega_j y^k(\bm{\mu}_j)}_{\mathbb{X}}^2
		\leq n_s n_p \sum_{k=1}^{n_t} \omega^k \sum_{j=1}^{n_s}  \omega_j^2 \norm{\hat{\bm{\mu}} - \bm{\mu}_j + \rho_j^k \bm{1}}_{\bm{C}_j^{k,1}}^2.
	\end{equation}
	
	As is well-documented in the literature (e.g., \cite[Remark~1.11]{bib:Gubisch2017}), term 2 is minimized by solving the eigenvalue problem
	\begin{equation}
		\label{eq:eigenvalueproblem}
		\bm{X}\bm{W}_{\hat{\bm{\mu}}} \bm{X}^T  \overline{\bm{\phi}}_i = \lambda_i \overline{\bm{\phi}}_i
	\end{equation}
	for the $n_r$ largest eigenvalues $\lambda_i$ and corresponding eigenvectors $\overline{\bm{\phi}}_i$. Here $\bm{X} = \bm{E}^{1/2}\bm{A}\bm{W}_k^{1/2}$, with $\bm{A}$ the snapshot matrix, $\bm{E} \in \mathbb{R}^{n_x \times n_x}$ the matrix corresponding to the inner product $\langle \cdot , \cdot \rangle_{\mathbb{X}}$, $\bm{W}_k \in \mathbb{R}^{n_s n_t \times n_s n_t}$ the weight matrix corresponding to the time integration (with entries $\left\{\omega^k\right\}_{k=1}^{n_t}$ on its diagonal), and $\bm{W}_{\hat{\bm{\mu}}} \in \mathbb{R}^{n_s n_t \times n_\omega n_t}$ the snapshot weight matrix with $\left\{\omega_j^2\right\}_{j=1}^{n_s}$ on its diagonal with $n_\omega$ the number of snapshots that have been assigned a nonzero weight. 
	Alternatively, the SVD of $\bm{X}_{\hat{\bm{\mu}}} = \bm{X} \bm{W}_{\hat{\bm{\mu}}}^{1/2}$ can be computed, as per \eqref{eq:standard_svd}.
	The largest $\text{min}(n_x, n_w n_t)$ singular values of $\bm{X}_{\hat{\bm{\mu}}}$ have a one-to-one correspondence to the absolute values of the largest eigenvalues of $\bm{X}_{\hat{\bm{\mu}}} \bm{X}_{\hat{\bm{\mu}}}^T$, such that $\lvert \lambda_i \rvert = \sigma_i^2$. The eigenvectors $\overline{\bm{\phi}}$ of $\bm{X}_{\hat{\bm{\mu}}}\bm{X}_{\hat{\bm{\mu}}}^T$ coincide with the left-singular vectors of $\bm{X}$ and are contained in its left-singular matrix $\bm{U}$. Term 2 is then bounded from above by the sum of the remaining eigenvalues, or equivalently, the squares of the remaining singular values, such that
	\begin{multline}
		\sum_{k=1}^{n_t} \omega^k \norm{\sum_{j=1}^{n_s} \omega_j y^k(\bm{\mu}_j) - \sum_{j=1}^{n_s} \sum_{i=1}^{n_r} \langle \omega_j y^k(\bm{\mu}_j), \phi_i \rangle_{\mathbb{X}} \phi_i}_{\mathbb{X}}^2\\
		\leq n_s \sum_{k=1}^{n_t} \omega^k \sum_{j=1}^{n_s} \omega_j^2 \norm{ y^k(\bm{\mu}_j) - \sum_{i=1}^{n_r} \langle y^k(\bm{\mu}_j), \phi_i \rangle_{\mathbb{X}} \phi_i}_{\mathbb{X}}^2
		\leq n_s \sum_{i = n_r + 1}^{n_s n_t} \sigma_i^2.
	\end{multline}
	
	For term 3, we remark that all terms are contained in $\text{span}\left\{\bm{\phi}_1, \ldots, \bm{\phi}_{n_r} \right\}$. Using the same partition of unity and averaged Taylor polynomial approach as for term 1, we find that
	\begin{equation*}
		\begin{aligned}
			\sum_{k=1}^{n_t} \omega^k \norm{\sum_{j=1}^{n_s} \sum_{i=1}^{n_r} \langle \omega_j y^k(\bm{\mu}_j), \phi_i \rangle_{\mathbb{X}} \phi_i - y_r^k(\hat{\bm{\mu}})}_{\mathbb{X}}^2 = & \sum_{k=1}^{n_t} \omega^k \norm{\sum_{j=1}^{n_s} \sum_{i=1}^{n_r} \langle \omega_j \tilde{R}_j^{1,k}, \phi_i \rangle_{\mathbb{X}} \phi_i}_{\mathbb{X}}^2 \\        
			\leq & n_s \sum_{k=1}^{n_t} \omega^k \sum_{j=1}^{n_s} \omega_j^2 \norm{\sum_{i=1}^{n_r} \langle \tilde{R}_j^{1,k}, \phi_i \rangle_{\mathbb{X}} \phi_i}_{\mathbb{X}}^2,
		\end{aligned}
	\end{equation*}
	which results in an upper bound that is of the same form as for term 1,
	\begin{multline}
		\sum_{k=1}^{n_t} \omega^k \norm{\sum_{j=1}^{n_s} \sum_{i=1}^{n_r} \langle \omega_j y^k(\bm{\mu}_j), \phi_i \rangle_{\mathbb{X}} \phi_i - y_r^k(\hat{\bm{\mu}})}_{\mathbb{X}}^2\\
		\leq n_s n_p \sum_{k=1}^{n_t} \omega^k \sum_{j=1}^{n_s} \omega_j^2 \norm{\hat{\bm{\mu}} - \bm{\mu}_j + \rho_j^{3,k} \bm{1}}_{\bm{C}_j^{k,3}}^2.
	\end{multline}
	Combining the upper bounds for terms 1, 2, and 3 through the triangle inequality of \eqref{eq:triangle_inequality} then results in \eqref{eq:PODerrorbound}.
\end{proof}
As shown in Theorem~\ref{th:errorbound}, the upper bound for the error between the HDM solution and its POD-Galerkin approximation at query point $\hat{\bm{\mu}}$ depends directly on the distances between $\hat{\bm{\mu}}$ and known snapshots at $\left\{\bm{\mu}_j \right\}_{j=1}^{n_s}$, as well as on the snapshot weights $\left\{\omega_j \right\}_{j=1}^{n_s}$. These snapshot weights can be chosen to (approximately) minimize the upper bound in Theorem~\ref{th:errorbound}. In the most general case, no information on the integration upper bound matrices $\bm{C}_j^{k,1}$ and $\bm{C}_j^{k,3}$, nor the ball radii $\rho_j^{1,k}$ and $\rho_j^{3,k}$ is available. In lieu of such information, we use the snapshot distance $2$-norm
\begin{equation}
	\delta_j(\hat{\bm{\mu}}) := \norm{\hat{\bm{\mu}} - \bm{\mu}_j}_2.
\end{equation}
We note that $\bm{\mu}_j \in \mathbb{P}$ is constant for a given $j$, while $\hat{\bm{\mu}}$ will vary depending on the current query point. We introduce
\begin{equation}
	\label{eq:distances}
	\begin{aligned}
		\delta_{\text{min}}(\hat{\bm{\mu}}) & := \min_{j = 1, \ldots, n_s} \delta_j(\hat{\bm{\mu}}), \\
		\delta_{\text{max}}(\hat{\bm{\mu}}) & := \max_{j = 1, \ldots, n_s} \delta_j(\hat{\bm{\mu}}), \\
		\hat{\delta} (c, \hat{\bm{\mu}}) & := c \delta_{\text{min}} (\hat{\bm{\mu}}) + (1-c) \delta_{\text{max}} (\hat{\bm{\mu}}),
	\end{aligned}
\end{equation}
where $\hat{\delta}$ is defined as linear combination of the minimum and maximum snapshot distances. We use $\hat{\delta}$ as cutoff distance: Any snapshot with a distance greater than $\hat{\delta}$ is assigned a weight of zero. With these definitions we can define versatile weighting functions that depend on only a single tunable parameter, $c \in (0, 1]$. In this work we focus on using a piecewise cubic polynomial weight function, in order to make the weights change in a continuous and differentiable manner. We did not observe a noticeable impact on the results presented in this work when using different polynomial weight functions, e.g. linear or quadratic. The general form of a suitable cubic polynomial $\omega \in C^1\left([\delta_{\text{min}}, \infty)\right)$ is
\begin{equation}
	\omega \left(\delta_j(\hat{\bm{\mu}})\right) := \begin{cases}
		1, & \text{if}\ \delta_j(\hat{\bm{\mu}}) = \delta_{\text{min}}(\hat{\bm{\mu}}), \\
		c_1 \delta_j^3(\hat{\bm{\mu}}) + c_2 \delta_j^2(\hat{\bm{\mu}}) + c_3 \delta_j(\hat{\bm{\mu}}) + c_4, & \text{if}\ \delta_{\text{min}}(\hat{\bm{\mu}}) < \delta(\hat{\bm{\mu}}) < \hat{\delta}(c,\hat{\bm{\mu}}), \\
		0, & \text{if}\ \delta(\hat{\bm{\mu}}) \geq \hat{\delta}(c, \hat{\bm{\mu}}).
	\end{cases}
\end{equation}
The coefficients $c_1$ through $c_4$ are uniquely defined with the conditions $\omega\left(\delta_{\text{min}}(\hat{\bm{\mu}})\right) = 1$, $\frac{{\rm d}\omega}{{\rm d}\delta }\left(\delta_{\text{min}}(\hat{\bm{\mu}})\right) = \omega\left(\hat{\delta}(c,\hat{\bm{\mu}})\right) = \frac{{\rm d}\omega}{{\rm d}\delta}\left(\hat{\delta}(\hat{\bm{\mu}})\right) = 0$. The snapshot weights $\left\{\omega_j \right\}_{j=1}^{n_s}$ follow with
\begin{equation}
	\omega_j(\hat{\bm{\mu}}) := \frac{\omega\left(\delta_j(\hat{\bm{\mu}})\right)}{\sum_{i=1}^{n_s} \omega\left(\delta_i(\hat{\bm{\mu}})\right)},
\end{equation}
which guarantees that the snapshot weights form a partition of unity. We calculate the derivatives of the weights with respect to changes in $\hat{\bm{\mu}}$. It follows from the chain rule that
\begin{equation}
	\label{eq:weightfunction_deriv}
	\frac{{\rm d} \omega_j}{{\rm d} \hat{\bm{\mu}}}(\hat{\bm{\mu}}) = \frac{\sum_{i=1}^{n_s} \omega\left(\delta_i(\hat{\bm{\mu}})\right) - \omega\left(\delta_j(\hat{\bm{\mu}})\right)}{\left( \sum_{i=1}^{n_s} \omega\left(\delta_i(\hat{\bm{\mu}})\right)\right)^2}\frac{\left(3c_1 \delta_j^2(\hat{\bm{\mu}}) + 2c_2 \delta_j(\hat{\bm{\mu}}) + c_3\right)}{\delta_j(\hat{\bm{\mu}})}(\hat{\bm{\mu}} - \bm{\mu}_{j}),
\end{equation}
with $\frac{{\rm d} \omega_j}{{\rm d} \hat{\bm{\mu}}}(\hat{\bm{\mu}}) \in \mathbb{R}^{n_p}$. We are interested in the derivative of $\omega_j$ in the direction $(\hat{\bm{\mu}} - \bm{\mu}_j)$. We further reduce the derivative vector to a single scalar variable by multiplying with the unit-norm vector in this direction.

\subsection{Efficiently applying snapshot weighting}
\label{subsec:snapshot_weighting}
As shown in the proof of Theorem~\ref{th:errorbound}, we weigh the columns of snapshot matrix $\bm{X}$ with the weights $\left\{\omega_j \right\}_{j=1}^{n_s}$ that populate $\bm{W}_{\hat{\bm{\mu}}}$ in order to tighten the POD approximation error. By leveraging \cref{alg:svd_weighted} we can efficiently compute the SVD of $\bm{X}_{\hat{\bm{\mu}}} = \bm{X}\bm{W}_{\hat{\bm{\mu}}}^{1/2}$ once we have the SVD of $\bm{X}$. Thus, computing the reduced basis corresponding to a specific set of snapshot weights is an additional step that follows after computing the global reduced basis. We reduce the size of $\bm{W}_{\hat{\bm{\mu}}}$ from $n_s n_t \times n_s n_t$ to $n_s n_t \times n_w n_t$, with $n_w < n_s$ the number of non-zero snapshot weights. In other words, columns corresponding to zero snapshot weights are removed from $\bm{W}_{\hat{\bm{\mu}}}$; we can do this without loss of accuracy or generality. Since the snapshots for which $\delta_j  \in [\delta_{\text{min}}, \hat{\delta})$ are the only ones that are assigned a nonzero weight, $n_w$ can be much smaller than $n_s$, depending on the snapshot distribution in $\mathbb{P}$. The cost of \Cref{alg:svd_weighted} is equal to the cost of three matrix multiplications plus the cost of computing the SVD of a matrix of size $n_s n_t \times n_w n_t$, which is $\mathcal{O}\left(n_s n_t (n_w n_t)^2\right)$. The cost of this SVD is thus independent of the number of HDM degrees of freedom $n_x$.

\begin{algorithm}
	\caption{Computing SVD of Weighted Snapshot Matrix} 
	\label{alg:svd_weighted}
	\begin{algorithmic}[1]
		\REQUIRE{SVD of snapshot matrix $\bm{X} = \bm{U}_{\bm{X}} \bm{\Sigma}_{\bm{X}} \bm{V}_{\bm{X}}^T \in \mathbb{R}^{n_x \times n_s n_t}$, weight matrix $\bm{W}_{\hat{\bm{\mu}}} \in \mathbb{R}^{n_s n_t \times n_w n_t}$}
		\ENSURE{SVD of weighted snapshot matrix $\bm{X}_{\hat{\bm{\mu}}} = \bm{X} \bm{W}_{\hat{\bm{\mu}}}^{1/2} = \bm{U} \bm{\Sigma} \bm{V}^T$}
		\STATE{Compute $\bm{D} := \bm{\Sigma}_{\bm{X}} \bm{V}_{\bm{X}}^T \bm{W}_{\hat{\bm{\mu}}}^{1/2}$}
		\STATE{Compute SVD: $\bm{D} = \bm{U}_{\bm{D}} \bm{\Sigma}_{\bm{D}} \bm{V}_{\bm{D}}^T$}
		\STATE{Set $\bm{V} := \bm{V}_{\bm{D}}$, $\bm{\Sigma} := \bm{\Sigma}_{\bm{D}}$, $\bm{U} := \bm{U}_{\bm{X}} \bm{U}_{\bm{D}}$}
	\end{algorithmic}
\end{algorithm}

\subsection{Applying parametric state derivatives}
\label{subsec:state_sensitivities}
Various works have shown that the derivatives of the state with respect to the design parameters can be used to make the reduced basis more robust when varying said parameters. First, the parametric snapshot derivatives can be added to the snapshot matrix \eqref{eq:snapshotmatrix}. Taylor expansion-based weight factors can be applied to them \cite{bib:Carlberg2008}, or each snapshot and derivative vector can be divided by the mean $L^2$ norm of the corresponding HDM solution snapshot \cite{bib:Schmidt2013}. Alternatively, the parametric derivatives can be collected in a separate snapshot matrix, and the SVD can be used to find their dominant modes directly \cite{bib:Zahr2015}. These can be appended to the dominant modes of the HDM solution snapshot matrix. Another option is to directly include the parametric derivatives in the reduced basis. Gram-Schmidt orthogonalization can then be applied to a matrix that also includes the state and adjoint-reduced bases \cite{bib:Wen2023} to arrive at an orthonormal basis.

In this work, we focus on the approach of Hay et al. \cite{bib:Hay2009}. Rather than adding the snapshot derivatives to $\bm{\Phi}$ or any snapshot matrix, they are instead used to compute parametric derivatives of the reduced basis vectors that make up $\bm{\Phi}$. The authors introduce two ways of using the parametric derivatives of the reduced basis vectors. In one approach, the derivatives are used to construct a first-order Taylor series extrapolation of the reduced basis, whereas in the other approach, the parametric derivatives of the leading reduced basis vectors are concatenated with the reduced basis itself. The performance of these two approaches is comparable when the reduced basis dimension is kept constant \cite{bib:Hay2009, bib:Jarvis2014}. The extrapolation approach requires computing the parametric derivatives of all $n_r$ columns of $\bm{\Phi}$, as opposed to only $n_r/2$ for the concatenation approach. Because of this, we focus here on the concatenation approach. The methods outlined in \cite{bib:Hay2009} are applicable only to one-dimensional parameter spaces and make use of HDM solution data at only a single parameter value. 

We generalize the aforementioned concatenation approach to efficient reduced basis derivative computations with HDM snapshots at multiple parameter values and multidimensional parameter spaces. To avoid having to compute reduced basis derivatives with respect to every design parameter, we determine for each solution snapshot $\bm{A}_j(\bm{\mu}_j)$ its derivative in the direction $\hat{\bm{\mu}} - \bm{\mu}_j$ pointing toward the new query point $\hat{\bm{\mu}}$. This approximation lets us construct a single snapshot derivative matrix independent of the dimension of $\mathbb{P}$. We use unit-norm vectors $\bm{\tau}^{(j)} = (\hat{\bm{\mu}} - \bm{\mu}_j) / \delta_j (\hat{\bm{\mu}})$ and multiply these with the derivatives $\frac{{\rm d} \bm{A}_j}{{\rm d}\mu_i}(\bm{\mu}_j) \in \mathbb{R}^{n_x \times n_t}$. We furthermore multiply with $\delta_j(\hat{\bm{\mu}})$ in order to normalize to $[0,1]$ the length of the line that connects $\bm{\mu}_j$ and $\hat{\bm{\mu}}$. The derivative of $\bm{A}_j(\bm{\mu}_j)$ in the direction $\hat{\bm{\mu}} - \bm{\mu}_j$ is
\begin{equation}
	\frac{{\rm d} \bm{A}_j}{{\rm d} (\hat{\bm{\mu}} - \bm{\mu}_j)}(\bm{\mu}_j) := \sum_{i=1}^{n_p}\frac{{\rm d} \bm{A}_j}{{\rm d} \mu_i}(\bm{\mu}_j) \tau_i^{(j)}\delta_j(\hat{\bm{\mu}}) = \sum_{i=1}^{n_p}\frac{{\rm d} \bm{A}_j}{{\rm d} \mu_i}(\bm{\mu}_j) \left[\hat{\mu}_i - \mu_{j,i}\right]
\end{equation}
for $j = 1, \dots, n_s$. In general, not every snapshot will have derivative data associated to it. Hence, we use only those snapshots for which we have accompanying derivative data. The snapshot weights $\left\{\omega_j \right\}_{j=1}^{n_s}$ are also computed with this reduced set of $n_{s,d}$ snapshots, such that other snapshots are always assigned a weight of zero. We combine the derivatives of multiple snapshots into the matrix
\begin{equation}
	\label{eq:snapshotderivmat}
	\frac{{\rm d}\bm{A}}{{\rm d}\hat{\bm{\mu}}}(\hat{\bm{\mu}}) := \begin{bmatrix} \frac{{\rm d} \bm{A}_1}{{\rm d} (\hat{\bm{\mu}} - \bm{\mu}_1)} & \frac{{\rm d} \bm{A}_2}{{\rm d} (\hat{\bm{\mu}} - \bm{\mu}_2)} & \cdots & \frac{{\rm d} \bm{A}_{n_s}}{{\rm d} (\hat{\bm{\mu}} - \bm{\mu}_{n_{s,d}})}\end{bmatrix}.
\end{equation}
We can thus readily recompute $\frac{{\rm d}\bm{A}}{{\rm d}\hat{\bm{\mu}}}(\hat{\bm{\mu}}) \in \mathbb{R}^{n_x\times n_s n_t}$ every time $\hat{\bm{\mu}}$ is updated. Recall that the reduced basis is computed through the SVD of the weighted snapshot matrix $\bm{X}_{\hat{\bm{\mu}}} = \bm{X} \bm{W}_{\hat{\bm{\mu}}}^{1/2}$. Applying the product rule,
\begin{equation}
	\label{eq:AW_productrule}
	\frac{{\rm d} \bm{X}_{\hat{\bm{\mu}}}}{{\rm d} \hat{\bm{\mu}}} = \bm{E}^{1/2}\frac{{\rm d} \bm{A}}{{\rm d} \hat{\bm{\mu}}} \bm{W}_k^{1/2}\bm{W}_{\hat{\bm{\mu}}}^{1/2} + \bm{E}^{1/2}\bm{A}\bm{W}_k^{1/2} \frac{{\rm d} \bm{W}_{\hat{\bm{\mu}}}^{1/2}}{{\rm d} \hat{\bm{\mu}}},
\end{equation}
where $\frac{{\rm d}\bm{A}}{{\rm d}\hat{\bm{\mu}}}$ is as in \eqref{eq:snapshotderivmat} and $\frac{{\rm d}\bm{W}_{\hat{\bm{\mu}}}^{1/2}}{{\rm d}\hat{\bm{\mu}}}$ follows from the derivative $\frac{{\rm d}\omega}{{\rm d}\hat{\bm{\mu}}}$ given in \eqref{eq:weightfunction_deriv}. The matrices $\frac{{\rm d} \bm{X}_{\hat{\bm{\mu}}}}{{\rm d} \hat{\bm{\mu}}}$ and $\bm{X}_{\hat{\bm{\mu}}}$ are the inputs to the algorithm to compute the reduced basis derivatives, as covered in \cite{bib:Hay2009}. This algorithm works by taking the derivative of eigenvalue problem \eqref{eq:eigenvalueproblem}; each reduced basis derivative is computed at the cost of solving a linear system of size $n_{w,d} n_t \times n_{w,d} n_t$, where $n_{w,d}$ is the number of snapshots with derivative information that are assigned a nonzero weight.

As shown in \cite{bib:Hay2009}, the derivatives $\frac{{\rm d}\bm{\phi}_i}{{\rm d}\hat{\bm{\mu}}}$ are orthogonal to $\bm{\phi}_i$. The changes to $\bm{\phi}_i$ that occur when moving away from  $\hat{\bm{\mu}} \in \mathbb{P}$ are thus not contained in the subspace spanned by the POD basis matrix $\bm{\Phi}$ that is constructed at $\hat{\bm{\mu}}$. Adding the derivatives of the reduced basis vectors to the reduced basis can thus make $\bm{\Phi}$ more accurate in a neighborhood around $\hat{\bm{\mu}}$. To obtain an enhanced reduced basis of dimension $n_r$ we retain the first $n_r / 2$ singular vectors; the remaining reduced basis dimensions are populated with the derivatives of these first $n_r / 2$ singular vectors. Lastly, we normalize the length of each reduced basis derivative vector before adding it to $\bm{\Phi}$.

\section{Optimization problems with POD}
\label{sec:optimization_with_POD}
The general form of PDE-constrained optimization problems after discretization can be stated as \begin{equation}
	\begin{aligned}
		\text{minimize} & & J(\bm{a}, \bm{\mu}) & \in \mathbb{R}\\
		\text{with respect to} & & \bm{a} & \in \mathbb{R}^{n_x n_t} \\
		& & \bm{\mu} & \in \mathbb{P} \subseteq \mathbb{R}^{n_p}\\
		\text{subject to} & & \bm{r}(\bm{a},\bm{\mu}) & = \bm{0} \\
		& & \bm{c}(\bm{a},\bm{\mu}) & \geq \bm{0}. \\
	\end{aligned}
	\label{eq:general_optimization}
\end{equation}
Suppose $J: \mathbb{R}^{n_x n_t} \times \mathbb{P} \rightarrow \mathbb{R}$ is an objective function that depends on the state variables $\bm{a} = \begin{bmatrix}
	\left(\bm{a}^1\right)^T & \cdots & \left(\bm{a}^{n_t}\right)^T
\end{bmatrix}^T \in \mathbb{R}^{n_x n_t}$ after discretization and the design parameters $\bm{\mu} \in \mathbb{P}$. We aim to minimize $J$ by varying the design parameters $\bm{\mu}$. Suppose that $\bm{r}(\bm{a},\bm{\mu}) = \begin{bmatrix}
	\left(\bm{r}^1(\bm{a}^1,\bm{\mu})\right)^T & \cdots & \left(\bm{r}^{n_t}(\bm{a}^{n_t},\bm{\mu})\right)^T
\end{bmatrix}^T \in \mathbb{R}^{n_x n_t}$ consists of the stacked residual vectors of all time steps.

The optimization problem \eqref{eq:general_optimization} can be solved in various ways. A broad overview of monolithic and distributed optimization architectures for single- and multidisciplinary design optimization is given in \cite{bib:Martins2013b}. We use a reduced-space approach in this work, in which the optimization algorithm and PDE solver are decoupled. The PDE solver is used to drive the residuals $\bm{r}$ to zero each time a new solution is requested by the optimization algorithm. This in contrast to full-space optimization approaches, wherein the optimization problem is expanded to include the PDE state variables as optimization parameters. The works \cite{bib:Joshy2023, bib:Joshy2021} unify these two optimization approaches and describe their respective advantages and downsides in more detail.

In reduced-space approaches, the optimization problem is solved by iteratively generating new query points $\bm{\mu}_i$ along a one-dimensional search direction. Newton's method is then used by the PDE solver to compute $\bm{r}(\bm{a}_i,\bm{\mu}_i) = \bm{0}$, which defines the new system state $\bm{a}_i$. The corresponding objective and active constraint function values follow accordingly, after which the optimization algorithm either chooses a new query point along the same search direction or picks a new search direction. In this work, we focus on gradient-based optimization methods, where the optimization algorithm uses derivative information to decide on new search directions. We use the OpenMDAO framework \cite{bib:Gray2019} with the sequential least-squares programming (SLSQP) gradient-based optimization algorithm \cite{bib:Kraft1988}. We use analytic derivative calculation methods \cite{bib:Martins2013}. When using POD we enforce
\begin{equation}
	\label{eq:POD_residual}
	\bm{\Phi}^T \bm{r}(\bm{\Phi}\bm{\Phi}^T\bm{a},\bm{\mu}) = \bm{0} 
\end{equation}
instead of enforcing $\bm{r}(\bm{a},\bm{\mu}) = \bm{0}$. The analytic derivative calculation methods that correspond to \eqref{eq:POD_residual} are derived in \cite{bib:Zahr2015}. While \cite{bib:Zahr2015} does not take into account any parameter dependence of $\bm{\Phi}$, computing the derivatives of $\bm{\Phi}$ with respect to all $n_p$ design parameters is not computationally tractable for high-dimensional parameter spaces. Hence, we do not include them in this work.

The cost of computing the required gradient information for all required outputs scales with the number of design parameters $n_p$ for the direct method, or the number of outputs $n_o$ for the adjoint method. The choice to use either method thus depends on the number of design parameters and outputs. We remark that the weighted POD approach proposed here is applicable to gradient-free optimization as well.

As was covered in \cref{subsec:state_sensitivities}, various works that use POD leverage the state derivatives $\frac{{\rm d}\bm{a}}{{\rm d}\bm{\mu}}$ as part of the process to construct $\bm{\Phi}$. These works focus on low-to-medium-dimensional parameter spaces, where the cost of the direct method is tractable even when the adjoint approach would be more efficient. While the direct method can be used in a reasonably efficient manner by resorting to multiple-right-hand-side solution approaches \cite{bib:Carlberg2008, bib:Carlberg2009, bib:Carlberg2011}, computing all $n_p$ state derivative vectors for large $n_p$ is significantly more expensive than using the adjoint method. Whether doing so is useful depends on how much the state derivative vectors aid the accuracy of any reduced model that might be used, and how often these state derivatives would need to be computed. 

\Cref{fig:pod_in_optimization_flowchart} shows our approach to deploying the POD approaches in an optimization problem. The dashed blocks are used by both weighted POD approaches, while the dotted block is used by the POD approach that uses state derivatives as covered in \cref{subsec:state_sensitivities}. We use a residual-based approach for \textit{a posteriori} error estimation each time a query point is simulated with POD. These are common in literature: The work \cite{bib:Feng2024} reviews various such error estimation methods for projection-based reduced models applied to parametric systems. We estimate the accuracy of the POD solution $\bm{a}_r^k$ in each time step by looking at the maximum $\mathbb{X}$-norm (over all time steps) of residual vector $\bm{r}^k$ per degree of freedom $n_x$ of the HDM. We define the residual convergence threshold $\epsilon_{\bm{r}} > 0$. A POD solution is accepted when $\epsilon_{\bm{r}} \geq \max \left\{||\bm{r}^k(\bm{\Phi}\bm{a}_{r}^k, \bm{\mu}) ||_{\mathbb{X}}/n_x: k = 1, 2, \ldots, n_t \right\}$, and rejected otherwise. If the POD solution is rejected, we rerun the simulation with the HDM. The SVD of the resulting HDM solution is computed, truncated if necessary, and appended to the SVD of the snapshot matrix with \Cref{alg:svd_update}. After possibly computing the required derivatives for the optimization algorithm, a new query point is chosen and a new iteration starts.

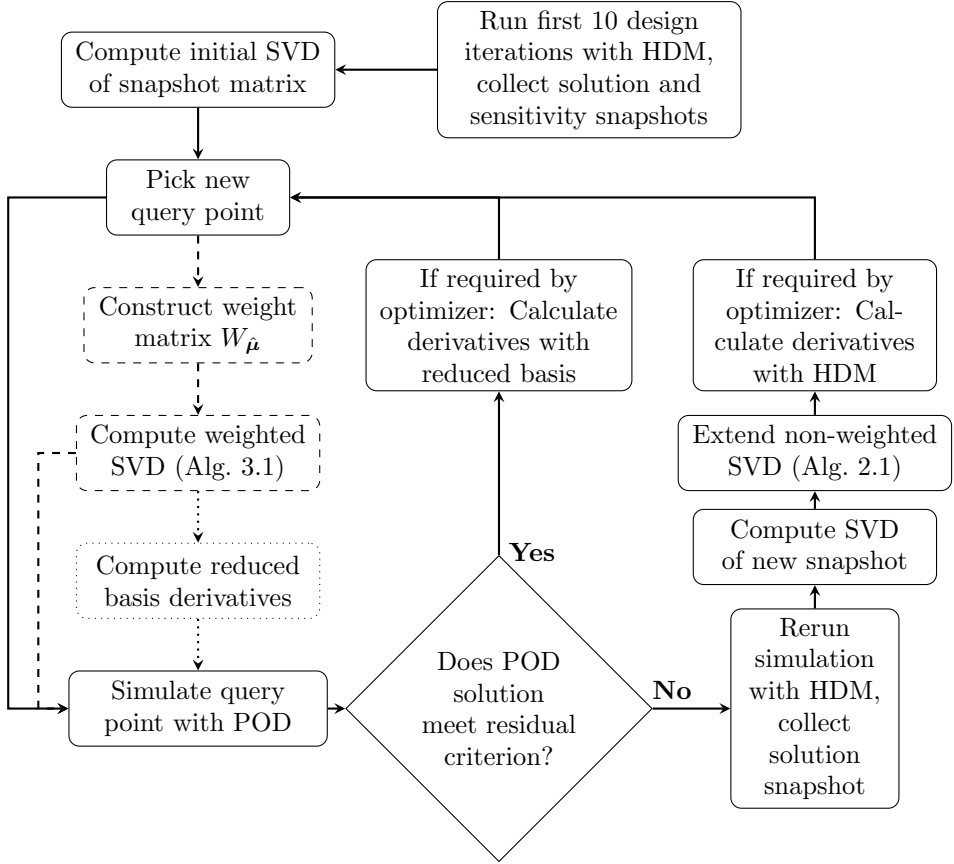
\begin{figure}[tbhp]
	\centering
	\begin{tikzpicture}[node distance=1.7cm]
		\node (start) [process, text width=3.8cm] {Run first $10$ design iterations with HDM, collect solution and sensitivity snapshots};
		\node (initial_SVD) [process, left of=start, text width=3.4cm, xshift=-3.5cm] {Compute initial SVD of snapshot matrix};
		\node (new_design) [process, below of=initial_SVD, text width=2.2cm] {Pick new query point};
		\node (weight_matrix) [process, below of=new_design, text width=2.8cm, dashed] {Construct weight matrix $W_{\hat{\bm{\mu}}}$};
		\node (weighted_SVD) [process, below of=weight_matrix, text width=3cm, dashed] {Compute weighted SVD (Alg.~\ref{alg:svd_weighted})};
		\node (RB_derivatives) [process, below of=weighted_SVD, text width=3cm, dotted] {Compute reduced basis derivatives};
		\node (POD) [process, below of=RB_derivatives, text width=3.2cm] {Simulate query point with POD};
		\node (converged) [decision, right of=POD, text width=2.1cm, xshift=2.3cm] {Does POD solution meet residual criterion?};
		\node (calc_derivs_rb) [process, text width=3.3cm] at (weight_matrix -| converged) {If required by optimizer: Calculate derivatives with reduced basis};
		
		\node (rerun_with_FM) [process, right of=converged, text width=2cm, xshift=2.5cm] {Rerun simulation with HDM, collect solution snapshot};
		
		\node (SVD_snapshot) [process, above of=rerun_with_FM, text width=3cm, yshift=0.45cm] {Compute SVD of new snapshot};
		\node (extend_SVD) [process, text width=3.4cm] at (weighted_SVD -| SVD_snapshot) {Extend non-weighted SVD (Alg.~\ref{alg:svd_update})};
		\node (calc_derivs) [process, text width=3cm] at (weight_matrix -| extend_SVD) {If required by optimizer: Calculate derivatives with HDM};
		
		\draw [arrow] (start) -- (initial_SVD);
		\draw [arrow] (initial_SVD) -- (new_design);
		\draw [arrow, dashed] (new_design) -- (weight_matrix);
		\draw [arrow, dashed] (weight_matrix) -- (weighted_SVD);
		\draw [arrow, dotted] (weighted_SVD) -- (RB_derivatives);
		\draw [arrow, dotted] (RB_derivatives) -- (POD);
		\draw [arrow] (POD) -- (converged);
		\draw [arrow] (converged) -- node [above, near start] {\textbf{No}} (rerun_with_FM);
		
		\draw [arrow] (rerun_with_FM) -- (SVD_snapshot);
		\draw [arrow] (SVD_snapshot) -- (extend_SVD);
		\draw [arrow] (extend_SVD) -- (calc_derivs);
		\draw [arrow] (calc_derivs) |- (new_design);
		\draw [arrow] (converged) -- node [right, near start, yshift=-0.5cm] {\textbf{Yes}} (calc_derivs_rb);
		\draw [arrow] (calc_derivs_rb) |- (new_design);
		\draw [arrow] (new_design.west) -- ++ (-1.3cm,0) |- (POD.west);
		\draw [arrow, dashed] (weighted_SVD.west) -- ++ (-0.5cm,0) |- (POD.west);
	\end{tikzpicture}    
	\caption{Overview of POD deployment while running optimization algorithm.}
	\label{fig:pod_in_optimization_flowchart}
\end{figure}

\section{Application: Dynamic shell thickness optimization}
\label{sec:numerical_experiments}

We use the methods outlined in the previous sections to minimize the compliance in a dynamic Reissner\textendash Mindlin (RM) shell model that is implemented in FEniCSx \cite{bib:Baratta2023}. This finite element shell model is used in \cite{bib:Xiang2024} to carry out time-dependent coupled aero-elastic simulations; a static version is used and compared to other shell and beam models in \cite{bib:Sarojini2024}. We consider the cantilevered shell of $10$ meters long and $2$ meters wide that is shown in \cref{fig:shell}. The shell is initially at rest, before being subjected to a constant pressure $\bm{f}$ in the negative vertical direction; no other loading is applied. This pressure force causes the shell to deflect in an oscillatory manner. The implicit midpoint rule is used, with time step size $\Delta t = 0.1$ seconds over a time domain of $20$ seconds; thus $n_t = 200$. At each time step, the shell model computes the degrees of freedom of the deformed state, $\bm{a}^k$, consisting of displacements $\bm{u}$ and rotations $\bm{\theta}$ such that $\bm{a}^k = \begin{bmatrix}
	(\bm{u}^k)^T & (\bm{\theta}^k)^T
\end{bmatrix}^T$, discretized with linear and quadratic basis functions respectively. We use a mesh with $20 \times 4$ quadrilateral cells, leading to $1{,}422$ degrees of freedom per time step. We use residual criterion $\epsilon_{\bm{r}} = 10^{-1}$ for each method, and we use $\mathbb{X} = H^1(\Omega)$ for the residual vector norm and POD basis construction. The $\mathbb{W}_2^1(I,\mathbb{X})$ norm thus consists of the (approximately) time-integrated squares of the instantaneous $H^1$ norms. All simulations are run on a laptop with an Intel Core i7-13620H processor and 16 GBs of RAM. 

\begin{figure}
	\centering
	\includegraphics[trim={380 185 230 270}, clip, width=0.7\linewidth]{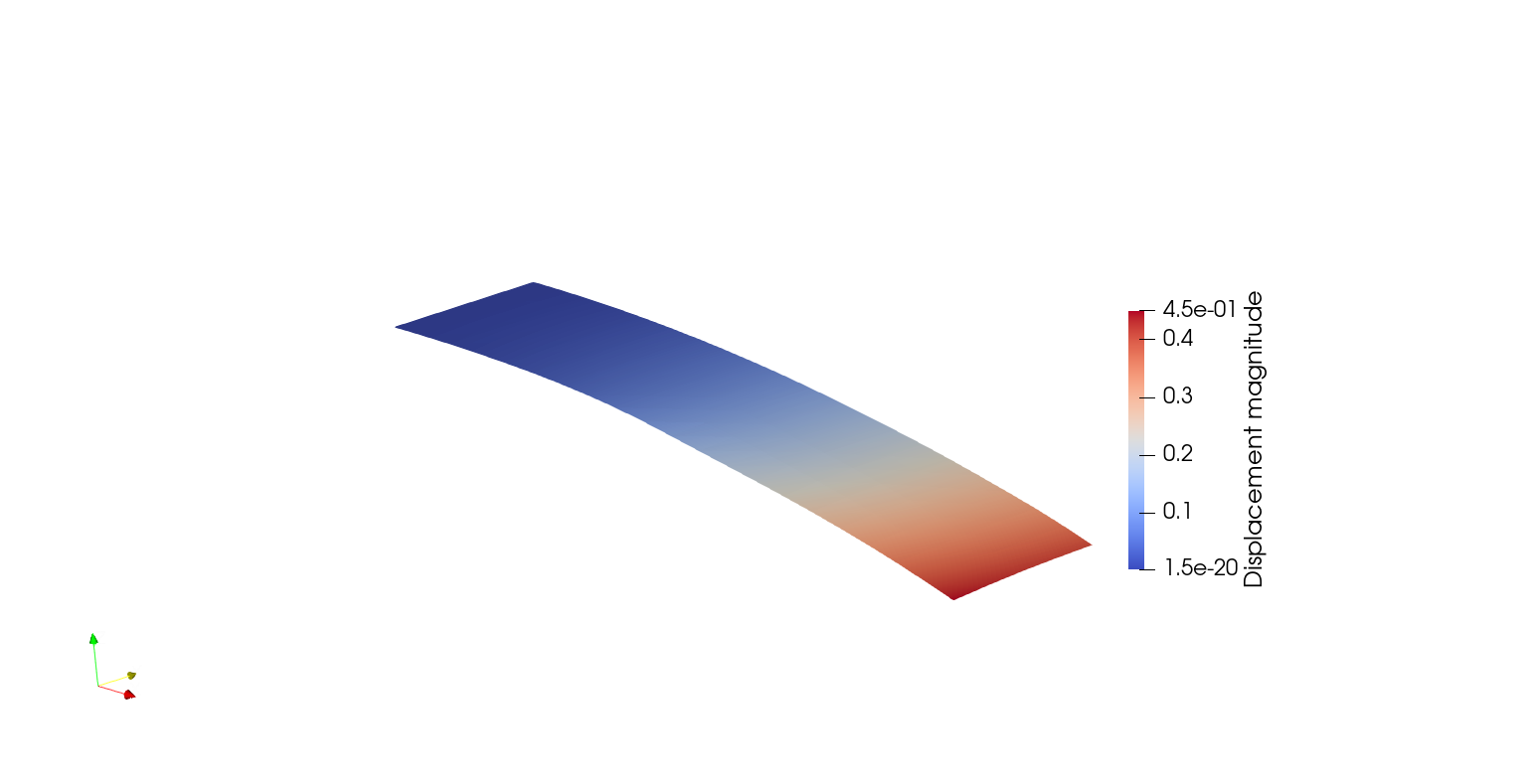}
	\caption{Cantilevered shell, deformation shown scaled with a factor of three}
	\label{fig:shell}
\end{figure}

\subsection{Optimization problem definition}
The objective of the optimization is to minimize the time integral, approximated with the trapezoid rule, of the instantaneous compliances $J^k$ over shell surface $\Omega$ over all time steps:
\begin{equation}
	J(\bm{a}, \bm{\mu}) := \int_0^T \int_\Omega \bm{u}(t, \bm{\mu})^T\cdot \bm{f}\ {\rm d}\Omega\ {\rm d}t \approx \sum_{k=1}^{n_t} \frac{\Delta t}{2} \int_\Omega \left(\bm{u}^{k-1}+\bm{u}^k\right)^T\cdot \bm{f}\ {\rm d}\Omega,
\end{equation}
where $\bm{u}^0$ is the displacement of the initial (undeformed) shell configuration. The shell thickness function is expanded in a linear function basis on the finite element mesh; we use the thickness degrees of freedom as design parameters, resulting in a total of $105$ design parameters for the $20 \times 4$ quadrilateral mesh. We impose a shell-volume equality constraint to prevent the optimizer from indefinitely increasing the shell thickness to minimize compliance: Starting from a uniform shell thickness of $0.1$ meters, we keep the shell volume $v: \mathbb{P} \rightarrow \mathbb{R}$ constant. Each thickness design parameter is allowed to vary within the range $\left[ 0.02, 0.2\right]$. This leads to the reduced-space optimization problem
\begin{equation}
	\begin{aligned}
		\text{minimize} & & J(\bm{a}(\bm{\mu}), \bm{\mu}) & \\
		\text{with respect to} & & \bm{\mu} & \in [0.02, 0.2]^{105} \\
		\text{subject to} & & v(\bm{\mu}) &= v(\bm{\mu}_0), \\
	\end{aligned}
\end{equation}
where $\bm{a}$ is treated as an implicit function of $\bm{\mu}$, and $\bm{\mu}_0$ is the initial design with a uniform thickness of $0.1\ \text{meters}$. When calculating derivatives for the optimization algorithm, there is only one output that depends on the PDE constraint, meaning that the adjoint method requires solving only a single linear system per time step. In contrast, using the direct method requires solving a linear system per time step for each thickness design parameter. The state sensitivities necessary for the POD approach outlined in \cref{subsec:state_sensitivities} are computed by inverting the matrix $\frac{\partial \bm{r}}{\partial \bm{a}} \in \mathbb{R}^{n_x \times n_x}$, as covered in \cref{sec:optimization_with_POD}. This is tractable for this linear time-dependent simulation since the matrix $\frac{\partial \bm{r}}{\partial \bm{a}}$ has to be inverted only once for each simulation, independent of the number of design parameters and time steps.

\subsection{POD approaches}
In order to compare the performance of the weighted POD methods that are proposed in this work we consider four POD approaches:
\begin{enumerate}
	\item a global reduced basis;
	\item interpolation in the tangent space to the Grassmann manifold \cite{bib:Amsallem2008};
	\item a reduced basis that leverages snapshot weighting, as per \cref{subsec:weighted_POD};
	\item a reduced basis that uses snapshot weighting as per \cref{subsec:weighted_POD}, and reduced basis derivatives as per \cref{subsec:state_sensitivities}.
\end{enumerate}
We use only the HDM for the first $10$ model evaluations of each optimization run, after which we start using the POD models. We use $c = 0.8$ for the cutoff distance $\hat{\delta}$ in \eqref{eq:distances} for all weight functions in the snapshot weighting approaches. For the Grassmann manifold interpolation, we use Gaussian radial basis function (RBF) interpolation through the SMT toolkit \cite{bib:Saves2024} with covariance kernel
\begin{equation}
	k_{\text{cov}}(\bm{\mu}_i, \bm{\mu}_j, \kappa) := \exp\left( -\frac{\norm{\bm{\mu}_i - \bm{\mu}_j}_2^2}{\kappa^2}\right).
\end{equation}
We only report here Grassmann manifold results for $\kappa = 0.1$; these were found to result in the most accurate reduced model. For all box plots, we truncate the whiskers at the $10^{\text{th}}$ and $90^{\text{th}}$ percentiles.

\subsection{Performance testing at identical query points}
\label{subsec:baseline_accuracytest}
We first solve the optimization problem with the HDM. This serves as a baseline. Following the initial $10$ HDM evaluations, we run every POD approach before each HDM evaluation, and record the corresponding POD approximation errors. This allows us to compare the accuracy of the POD approaches when using the same snapshot data. After each HDM evaluation (or derivative evaluation), we collect the corresponding data and use it to update the POD models. 

\Cref{fig:error_boxplots_followingFOM} shows box plots of the resulting relative POD approximation errors corresponding to each method for reduced basis sizes $n_r \in \left\{30, 40, 50, 60 \right\}$ with the global and weighted POD approaches. The corresponding Grassmann reduced basis size is constant and equal to $n_r = 15$. We use the $\mathbb{W}_2^1(I,\mathbb{X})$ norm that is defined in \eqref{eq:time_integral_norm} to quantify the POD approximation error, and normalize it with the $\mathbb{W}_2^1(I,\mathbb{X})$ norm of the corresponding HDM solution. The Grassmann manifold performs comparably to the global POD basis with $n_r = 50$, although its approximation errors show more variance. Both approaches are outperformed by the weighted POD approaches with and without use of reduced basis derivatives: Applying snapshot weighting results in a median error reduction of two orders of magnitude for the same reduced basis size. In contrast, with weighting and reduced basis derivatives, we observe a consistent reduction between four and five orders of magnitude relative to the global POD basis.

\begin{figure}[tbhp]
	\centering
	\begin{subfigure}[t]{0.38\textwidth}
		\centering
		\includegraphics[height=4.8cm]{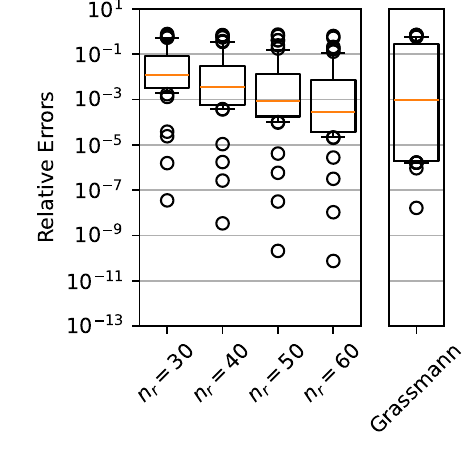}
		\caption{Global and Grassmann}
		\label{fig:errors_global_basis_followingFOM}
	\end{subfigure}
	\hfill
	\begin{subfigure}[t]{0.3\textwidth}
		\centering
		\includegraphics[height=4.8cm]{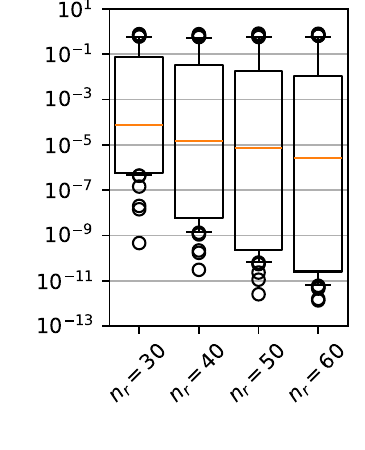}
		\caption{Weighted}
		\label{fig:errors_weighted_basis_followingFOM}
	\end{subfigure}
	\hfill
	\begin{subfigure}[t]{0.3\textwidth}
		\centering
		\includegraphics[height=4.8cm]{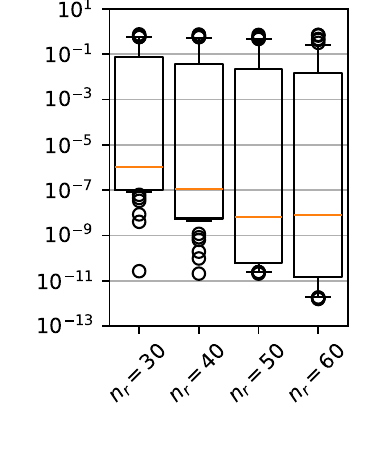}
		\caption{Weighted + derivatives}
		\label{fig:errors_sensitivities_basis_followingFOM}
	\end{subfigure}
	\caption{Relative $\mathbb{W}_2^1(I,\mathbb{X})$ norm error box plots of each POD approach when testing the POD approaches every time the HDM is queried. Each POD approach is tested at the same set of parameter values, with the same HDM solution data.}
	\label{fig:error_boxplots_followingFOM}
\end{figure}

\Cref{fig:walltime_boxplots_followingFOM} shows the total wall times per model evaluation corresponding to the error box plots of \cref{fig:error_boxplots_followingFOM}. This includes the time to update the SVD of the snapshot matrix if necessary, and all steps needed to construct the reduced bases, as well as the simulation wall time. The wall time numbers reported in this work are those of individual test runs; repeated testing showed negligible wall time variations. The black horizontal lines shown in each graph indicate the average simulation time of the HDM. The Grassmann manifold approach has the shortest wall time, owing to its lack of a global snapshot matrix and associated SVD calculation. As the reduced basis sizes of the other POD approaches are increased, their wall times increase accordingly. The wall times of the global POD approach and the weighted approach without derivatives are comparable, despite the latter computing the SVD of the weighted snapshot matrix for each set of snapshots weights. This shows that this extra basis adaptation step can be carried out efficiently. The median wall times of the weighted approach with snapshot derivatives are similar to those of the global approach and weighted approach without derivatives for $n_r \leq 50$, and are lower with $n_r = 60$. The reduced basis has an effect on the number of Newton iterations that are required to obtain a suitably converged solution. This is reflected in the wall time spread of the weighted approach with derivatives, as well as the low wall time outliers of the weighted approach without derivatives. The main conclusion that can be drawn from \cref{fig:walltime_boxplots_followingFOM} is that the extra reduced basis computation steps do not result in prohibitively expensive pre-processing computation times.

\begin{figure}[tbhp]
	\centering
	\begin{subfigure}[t]{0.38\textwidth}
		\centering
		\includegraphics[height=4.8cm]{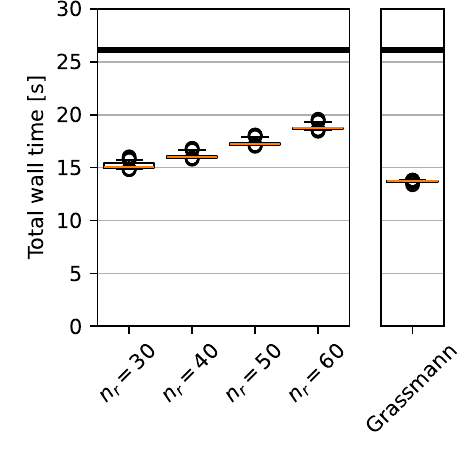}
		\caption{Global and Grassmann}
	\end{subfigure}
	\hfill
	\begin{subfigure}[t]{0.3\textwidth}
		\centering
		\includegraphics[height=4.8cm]{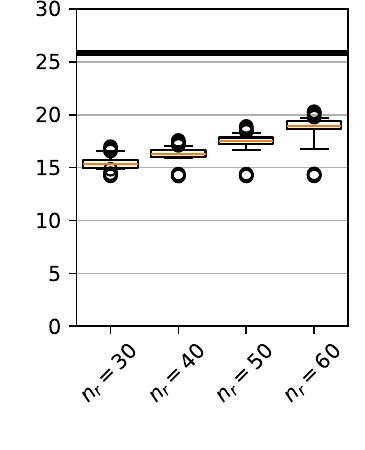}
		\caption{Weighted}
	\end{subfigure}
	\hfill
	\begin{subfigure}[t]{0.3\textwidth}
		\centering
		\includegraphics[height=4.8cm]{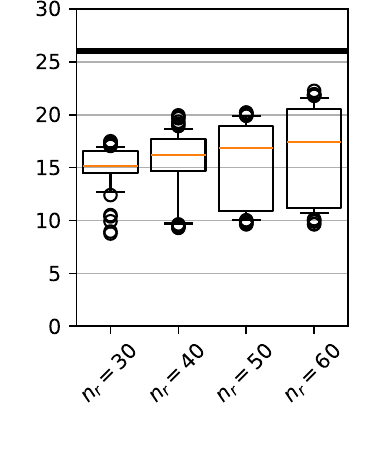}
		\caption{Weighted + derivatives}
		\label{fig:walltime_sensitivities_followingFOM}
	\end{subfigure}
	\caption{Wall time box plots of each POD approach when testing the POD approaches every time the HDM is queried. Each POD approach is tested at the same set of parameter values, with the same HDM solution data.}
	\label{fig:walltime_boxplots_followingFOM}
\end{figure}

\Cref{fig:medianerrors_vs_walltimes_followingFOM} shows the median wall times per model evaluation graphed against the median relative errors for the data sets shown in \cref{fig:error_boxplots_followingFOM} and \cref{fig:walltime_boxplots_followingFOM}. The vertical black line indicates the average HDM wall time. We observe that the global POD approach is outperformed by all three other methods, either by being less accurate or having a higher wall time for the same level of accuracy. The Grassmann manifold approach is the fastest of the methods considered, but is also several orders of magnitude less accurate than the weighted approaches with and without derivative information. Moreover, the weighted approach with derivative information outperforms the weighted approach without derivative information: Including derivative information reduces the relative error by at least an order of magnitude when keeping the median wall time constant.

\begin{figure}
	\centering
	\includegraphics[width=0.95\linewidth]{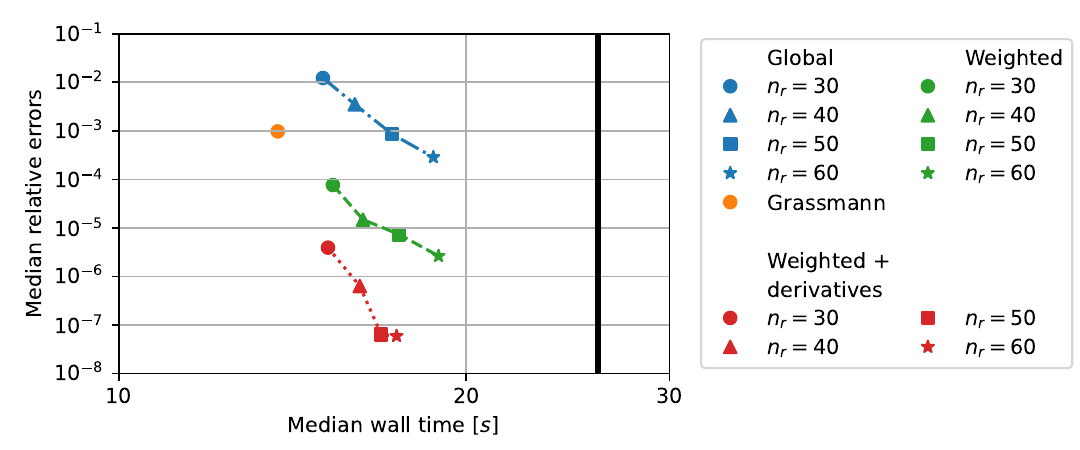}
	\caption{Median wall times versus median relative errors per POD evaluation when testing the POD approaches each time the HDM is queried.}
	\label{fig:medianerrors_vs_walltimes_followingFOM}
\end{figure}

\subsection{Optimization with the POD approaches}
\label{subsec:POD_testing}
In the second set of numerical experiments, we let the optimization algorithm make use of the POD solutions. We split the optimization procedure into multiple steps: Every $25$ SLSQP iterations, we re-initialize the optimization from the best-performing design, thereby resetting the Hessian approximation. At each re-initialization, we also reduce $\epsilon_{\bm{r}}$ to half of its previous value, thereby increasing the required accuracy of the POD-reduced model, and add a solution snapshot of the aforementioned best-performing design to the collection of snapshots if not included already. This approach allows us to accept more inaccurate POD solutions during the initial stages of the optimization process while still obtaining accurate approximate solutions near any optimal design, thereby trading off design space exploration and exploitation. To ensure sufficient POD accuracy near the optimal design point, we require $\epsilon_{\bm{r}}$ to be halved at least two times; in other words, $\epsilon_{\bm{r}}$ is to be decreased to at most one-fourths of its original value, and potentially less if more SLSQP iterations are needed to find an optimal design. We emphasize that each POD approach is used in separate runs of the same optimization problem: Based on the predictions of the POD models the optimization algorithm takes different paths through the parameter space, which results in different local optimal designs and required numbers of design iterations.

\begin{table}[tbhp]
	\footnotesize
	\caption{Number of (state) snapshots $n_s$, derivative snapshots $n_{s,d}$, successful (accepted) POD evaluations at unique query points $n_{a}$, model evaluations at unique query points $n_{e}$, and relative optimized compliance $C_{\text{rel}} = C_{\text{optimized}}/C_{\text{initial}}$ with various basis sizes $n_r$.}
	\label{tab:outputs}
	\subfloat[Global]{
		\begin{tabular}{ p{0.5cm}|p{0.5cm}p{0.5cm}p{0.5cm}p{0.85cm} }
			$n_r$ & $n_s$ & $n_{a}$ & $n_{e}$ & $C_{\text{rel}}$ \\ \hline
			$30$ & 181 & 2 & 183 & 0.2129 \\
			$40$ & 97 & 4 & 100 & 0.1972 \\
			$50$ & 119 & 14 & 132 & 0.2270 \\
			$60$ & 159 & 4 & 163 & 0.2194 \\
	\end{tabular}}
	\hfill
	\subfloat[Grassmann manifold]{\begin{tabular}{ p{0.8cm}|p{0.5cm}p{0.5cm}p{0.5cm}p{0.85cm} }
			$\kappa$ & $n_s$ & $n_{a}$ & $n_{e}$ & $C_{\text{rel}}$  \\ \hline
			$0.1$ & 124 & 97 & 220 & 0.2011  \\
	\end{tabular}}
	
	\subfloat[Weighted]{\begin{tabular}{ p{0.5cm}|p{0.5cm}p{0.5cm}p{0.5cm}p{0.85cm} }
			$n_r$ & $n_s$ & $n_{a}$ & $n_{e}$ & $C_{\text{rel}}$ \\ \hline
			$30$ & 87  & 28 & 113 & 0.1979 \\
			$40$ & 49 & 143 & 189 & 0.1961 \\
			$50$ & 39 & 36 & 73 & 0.1961 \\
			$60$ & 50 & 53 & 101 & 0.1960 \\
	\end{tabular}}
	\hfill
	\subfloat[Weighted + derivatives]{\begin{tabular}{ p{0.5cm}|p{0.5cm}p{0.6cm}p{0.5cm}p{0.5cm}p{0.85cm} }
			$n_r$ & $n_s$ & $n_{s,d}$ & $n_{a}$ & $n_{e}$ & $C_{\text{rel}}$ \\ \hline
			$30$ & 167 & 73 & 98 & 262 & 0.1978 \\
			$40$ & 67 & 24 & 60 & 126 & 0.1974 \\
			$50$ & 55 & 26 & 56 & 109 & 0.2083 \\
			$60$ & 30 & 14 & 29 & 57 & 0.1978 \\
	\end{tabular}}
	
\end{table}

\Cref{tab:outputs} summarizes the optimization results obtained with each method while using the same reduced basis sizes as in \cref{subsec:baseline_accuracytest}. For most test runs the number of snapshots plus the number of accepted POD evaluations exceeds the number of model evaluations at unique query points. This happens when a previous best-performing design is queried twice: At the initial query the POD solution is accepted, but the second query happens at the start of a re-initialized optimization run, where we require a snapshot to be saved. We note that the optimization runs with each method converge to a variety of local optima. The global POD basis consistently meets the convergence threshold in only a few iterations, as evidenced by the low number $n_a$ of accepted POD solutions, thereby providing nearly no speedup compared to only using the HDM. While the Grassmann manifold interpolation approach fares better, it still requires a high number of snapshots $n_s$. This likely happens due to errors when approximating solution gradients, which then cause the optimization algorithm to need more iterations to converge to a local optimal design. With weighted POD without derivative information we consistently find that fewer solution snapshots are needed than for global POD and the Grassmann manifold optimization run. Including reduced basis derivatives in the POD basis leads to an increase in the number of required HDM evaluations over the weighted approach without derivatives for most considered $n_r$. This happens because we are only able to make use of the solution snapshots that have a corresponding derivative snapshot; thus the number of available snapshots for constructing the reduced basis is equal to $n_{s,d}$, and not $n_s$. The difference between $n_s$ and $n_{s,d}$ suggests that the weighted approach with derivatives could benefit from using another optimization algorithm to lower $n_s$ further, such as a trust-region-based algorithm instead of SLSQP's line search-based approach. We furthermore see that the total number of model evaluations $n_e$ for the weighted approach with derivatives decreases as the reduced basis dimension is increased.

\begin{figure}[tbhp]
	\centering
	\begin{subfigure}[t]{0.38\textwidth}
		\centering
		\includegraphics[height=4.8cm]{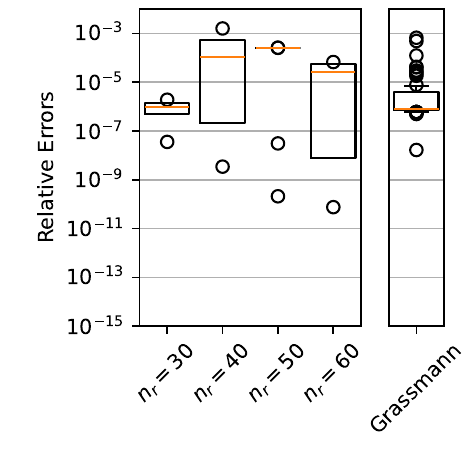}
		\caption{Global and Grassmann}
		\label{fig:errors_global_basis}
	\end{subfigure}
	\hfill
	\begin{subfigure}[t]{0.3\textwidth}
		\centering
		\includegraphics[height=4.8cm]{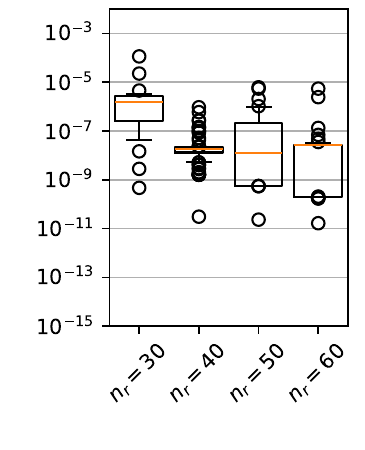}
		\caption{Weighted}
		\label{fig:errors_weighted_basis}
	\end{subfigure}
	\hfill
	\begin{subfigure}[t]{0.3\textwidth}
		\centering
		\includegraphics[height=4.8cm]{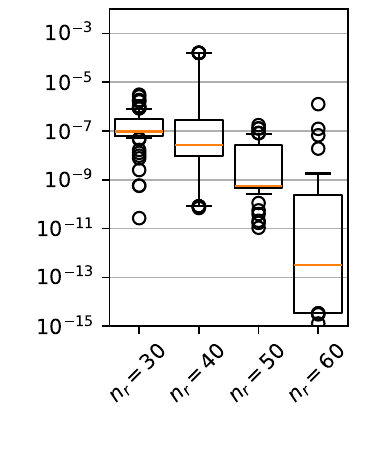}
		\caption{Weighted + derivatives}
		\label{fig:errors_sensitivities_basis}
	\end{subfigure}
	\caption{Relative $\mathbb{W}_2^1(I,\mathbb{X})$ norm error box plots of each POD approach for POD solutions accepted based on the residual norm tolerance $\epsilon_{\bm{r}}$.}
	\label{fig:error_boxplots}
\end{figure}

\begin{figure}[tbhp]
	\centering
	\begin{subfigure}[t]{0.38\textwidth}
		\centering
		\includegraphics[height=4.8cm]{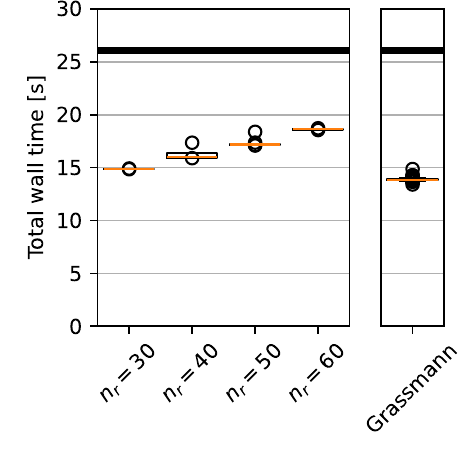}
		\caption{Global and Grassmann}
	\end{subfigure}
	\hfill
	\begin{subfigure}[t]{0.3\textwidth}
		\centering
		\includegraphics[height=4.8cm]{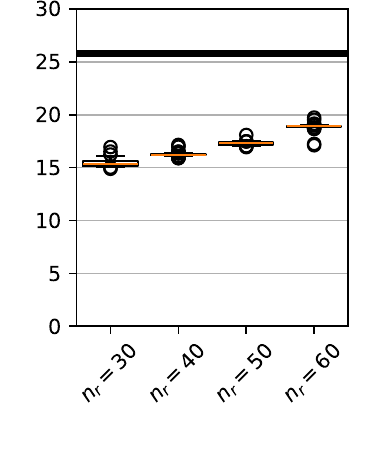}
		\caption{Weighted}
	\end{subfigure}
	\hfill
	\begin{subfigure}[t]{0.3\textwidth}
		\centering
		\includegraphics[height=4.8cm]{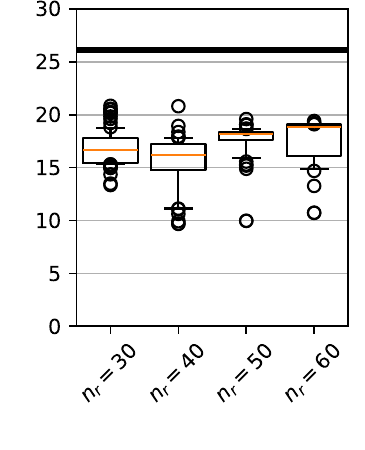}
		\caption{Weighted + derivatives}
		\label{fig:walltime_sensitivities}
	\end{subfigure}
	\caption{Wall time box plots of each POD approach for POD solutions accepted based on the residual norm tolerance $\epsilon_{\bm{r}}$.}
	\label{fig:walltime_boxplots}
\end{figure}

\Cref{fig:error_boxplots} shows box plots of the relative error norms of every POD solution that satisfies the residual criterion $\epsilon_{\bm{r}}$. We emphasize again that each optimization run follows its own trajectory through the parameter space with its own set of snapshot data, depending on which previous POD solutions did not satisfy the residual criterion. As mentioned above, the global POD basis has only a handful of accepted POD evaluations, leading to only a handful of error data points in the box plot, without any visible error convergence when increasing the size of the reduced basis. The Grassmann manifold-based approach has a median relative error around $10^{-6}$. This is comparable to the weighted approach with $n_r = 30$. When using weighted POD without derivatives, we see a consistent decrease in the median error between $n_r = 30$ and larger reduced bases. We also see that the weighted POD approach is again multiple orders of magnitude more accurate than both the global POD approach and the Grassmann manifold interpolation approach, especially for $n_r > 30$. This highlights an advantage of the weighted approach over the Grassman manifold-based approach for creating adaptive reduced bases: We are able to increase $n_r$ beyond the size of any local reduced basis, whereas with the Grassmann manifold-based approach we are limited to $n_r \leq 15$. Next, we look at the weighted POD approach with reduced basis derivatives. We see a decrease in the median error as the reduced basis size is increased. For $n_r < 50$ the accuracy of this method seems to be comparable to the weighted POD approach without reduced basis derivatives, while using fewer solution snapshots. With larger reduced basis sizes the median error keeps decreasing, even while the number of required snapshots with derivative information $n_{s,d}$ and the required number of model queries decrease accordingly. 

Lastly, \cref{fig:walltime_boxplots} shows the total wall time of each POD solution that meets the residual convergence threshold. We point out that these wall times are representative of the actual deployment of each POD approach. The trends shown here are identical to those of \cref{fig:walltime_boxplots_followingFOM}, where the wall times of the weighted approach without derivatives are similar to those of the global approach for all $n_r$. The Grassmann manifold-based approach is slightly faster than the other methods. The weighted approach with derivatives shows more fluctuation in wall time, with the median wall times being comparable to the aforementioned methods. This is despite having the smallest approximation errors of all methods considered here.

\section{Conclusions}
We presented methods for using adaptive, weighted proper orthogonal decomposition (POD) for high-dimensional optimization. We proposed here an efficient algorithm for computing local reduced bases for any new query point in the parameter space, by weighting the columns of the snapshot matrix. We prove an upper bound for the POD approximation error for time-dependent parametric models, and show how weighting the columns of the snapshot matrix affects this error bound. Choosing these snapshot weights based on their proximity to the current query point in the parameter space is shown to tighten the error bound, leading to a heuristic for weighted POD. We furthermore generalize the approach introduced in \cite{bib:Hay2009} to compute derivatives of the reduced basis in multidimensional parameter spaces, and show how snapshot weighting is a natural generalization of this approach whenever snapshot data is available for multiple points in the parameter space.

The methods presented are applied to a design optimization problem with $105$ design parameters and a linear time-dependent model with $200$ time steps and $1{,}422$ spatial degrees of freedom. Weighted POD approaches with and without derivative information are shown to be at least two orders of magnitude more accurate when using the same snapshot data as a global POD basis and a reference implementation of the Grassmann manifold interpolation approach. The computational cost of the weighted POD approach is shown in a numerical experiment to be comparable to that of using a global POD basis. While the online cost of computing reduced basis derivatives can be prohibitive in the presence of large amounts of data, the high levels of accuracy of this method allow for the use of sparse high-dimensional data, thereby limiting its computational overhead.

There are several ways in which these results could be improved upon. First, in this work, we used the Euclidean norm to measure distances between points in the parameter space. A more appropriate distance measure should be more sensitive to design parameters that lead to larger changes in the system state; this is also reflected in the POD error bound. One way of constructing such a distance measure could be with the active subspace method \cite{bib:Constantine2015}. Second, to obtain significant speedups a hyper-reduction approach should be incorporated in the presented methods. The discrete empirical interpolation method \cite{bib:Chaturantabut2009} and energy-conserving sampling and weighting \cite{bib:Farhat2015} can both make use of the weighted POD approach presented here. Alternatively, non-intrusive model learning approaches like operator inference \cite{bib:Peherstorfer2016} and dynamic mode decomposition \cite{bib:Schmid2010} also use POD for dimension reduction, and could thus make use of the weighted POD approaches presented here. Third, we used a simple way of embedding reduced models into an optimization algorithm by re-initializing the optimization after every $25$ SLSQP iterations, while also tightening the POD residual convergence threshold. Using a trust-region optimization approach like in \cite{bib:Wen2023, bib:Zahr2015} could lead to more effective and efficient use of the POD models. Alternatively, a full-space or hybrid optimization approach \cite{bib:Joshy2023, bib:Joshy2021} could be used; this could allow for more efficient use of reduced models by making the residual convergence threshold adaptive without having to re-initialize the optimization.

\bibliographystyle{siamplain}
\bibliography{sample}
\end{document}